\documentclass[a4paper, 11pt, reqno]{amsart}

\usepackage{amsmath,amssymb,amsthm}
\usepackage{enumerate}
\usepackage{xcolor}
\usepackage[mathscr]{euscript}
\usepackage{url}
\usepackage[all]{xy}
\usepackage{booktabs}
\usepackage{graphicx}
\usepackage{algorithmic}
\usepackage[subrefformat=parens, labelfont=up]{subcaption}
\usepackage{mathtools}
\usepackage{tablefootnote}
\usepackage{makecell}
\usepackage{multirow}
\usepackage{enumitem}
\usepackage{longtable}

\usepackage{hyperref}
\definecolor{cobalt}{rgb}{0.0, 0.28, 0.67}

\hypersetup
{
bookmarksopen=true,
pdftitle={Globally Convergent Jacobi-Type Algorithms for Robust Approximate Diagonalization of Higher-Order Cumulant Tensors},
pdfauthor={Jianze Li, Zhenyu Cai, Wentao Ding, Konstantin Usevich},
pdfsubject={Nonsmooth Jacobi-type algorithms for cumulant tensor diagonalization},
pdfmenubar=true,
pdfhighlight=/O,
colorlinks=true,
pdfpagemode=UseNone,
pdfpagelayout=SinglePage,
pdffitwindow=true,
linkcolor=cobalt,
citecolor=cobalt,
urlcolor=cobalt
}

\usepackage{aliascnt}
\usepackage[capitalize,nameinlink,noabbrev]{cleveref}

\usepackage[scale=0.75]{geometry}
\usepackage{tikz}
\usetikzlibrary{shapes}
\usetikzlibrary{arrows}
\usetikzlibrary{matrix}

\theoremstyle{plain}
\newtheorem{theorem}{Theorem}[section]
\newcommand{\newsharedtheorem}[3]{%
  \newaliascnt{#1}{theorem}%
  \newtheorem{#1}[#1]{#2}%
  \aliascntresetthe{#1}%
  \crefname{#1}{#2}{#3}%
  \Crefname{#1}{#2}{#3}}
\newsharedtheorem{corollary}{Corollary}{Corollaries}
\newsharedtheorem{lemma}{Lemma}{Lemmas}
\newsharedtheorem{proposition}{Proposition}{Propositions}

\theoremstyle{definition}
\newsharedtheorem{definition}{Definition}{Definitions}
\newsharedtheorem{problem}{Problem}{Problems}

\newsharedtheorem{example}{Example}{Examples}
\newsharedtheorem{notation}{Notation}{Notations}

\theoremstyle{remark}
\newsharedtheorem{remark}{Remark}{Remarks}

\theoremstyle{definition}

\usepackage{algorithm}

\newcommand{\tens}[1]{\boldsymbol{\mathcal{#1}}}
\newcommand{\tenselem}[1]{\mathcal{#1}}

\newcommand{\matr}[1]{\boldsymbol{#1}}

\newcommand{\vect}[1]{\boldsymbol{#1}}

\newcommand{\set}[1]{\mathcal{#1}}
\newcommand{\T}{{\sf T}}        
\renewcommand{\H}{{\sf H}}      

\makeatletter
\newcommand{\rank}[1]{\mathop{\operator@font rank}(#1)}
\newcommand{\colrank}[1]{\mathop{\operator@font colrank}\{#1\}}
\newcommand{\krank}[1]{\mathop{\operator@font krank}\{#1\}}
\newcommand{\trace}[1]{\mathop{\operator@font tr}\left(#1\right)}
\newcommand{\symmm}[1]{\mathop{\operator@font \bf symm}\left(#1\right)}
\newcommand{\symmo}[1]{\mathop{\operator@font sym}\left(#1\right)}
\newcommand{\skeww}[1]{\mathop{\operator@font skew}\left(#1\right)}
\newcommand{\skewww}[1]{\mathop{\operator@font \bf skew}\left(#1\right)}
\newcommand{\Diag}[1]{\mathop{\operator@font\bf Diag}\{#1\}}    
\newcommand{\diag}[1]{\textbf{\operator@font\bf diag}\left(#1\right)}    
\newcommand{\diagg}[2]{\mathop{\operator@font diag}_{#2}\{#1\}}    
\newcommand{\Span}[1]{\mathop{\operator@font Span}\{#1\}}    
\newcommand{\argmin}{\mathop{\operator@font argmin}}

\newcommand{\offdiag}[1]{\mathop{\operator@font offdiag}\{#1\}}    
\newcommand{\Proj}[2]{\mathop{\operator@font Proj_{#1}}{#2}}
\newcommand{\ProjGrad}[2]{\mathop{{\operator@font grad} }#1(#2)}

\newcommand{\expp}[1]{\mathop{\operator@font exp}\left(#1\right)}
\makeatother

\newcommand{\eqdef}{\stackrel{\sf def}{=}}
\newcommand{\RR}{\mathbb{R}}
\newcommand{\CC}{\mathbb{C}}

\newcommand{\Sphere}[1]{\set{S}^{#1}}
\newcommand{\TT}{\mathbb{T}}

\renewcommand{\Re}{\operatorname{Re}}
\renewcommand{\Im}{\operatorname{Im}}

\newcommand{\ON}[1]{\set{O}_{#1}}

\newcommand{\UN}[1]{\set{U}_{#1}}

\newcommand{\Gmat}[3]{\matr{G}^{(#1,#2,#3)}}

\newcommand{\contr}[1]{\mathop{\bullet_{#1}}}   

\newcommand{\subalgref}[1]{\hyperref[#1]{Subalgorithm~\ref*{#1}}}

\DeclareMathOperator*{\argmax}{\arg\!\max}

\newcommand{\grad}{\operatorname{grad}} 
\DeclareMathOperator{\dist}{dist}

\begin{document}

\title[Globally Convergent Jacobi-Type Algorithms]
{Globally Convergent Jacobi-Type Algorithms for Robust Approximate
Diagonalization of Higher-Order Cumulant Tensors}

\author{Jianze Li}
\address{School of Science, Sun Yat-sen University,
Shenzhen, Guangdong, China}
\email{lijianze@mail.sysu.edu.cn}

\author{Zhenyu Cai}
\address{School of Science, Sun Yat-sen University,
Shenzhen, Guangdong, China}
\email{caizhy27@mail2.sysu.edu.cn}

\author{Wentao Ding}
\address{Department of Systems Engineering and Engineering Management,
The Chinese University of Hong Kong, Shatin, N.T., Hong Kong}
\email{wentaoding@cuhk.edu.hk}

\author{Konstantin Usevich}
\address{Centre de Recherche en Automatique de Nancy (CRAN), Université de Lorraine, CNRS,
Campus Sciences, 54506 Vandœuvre-lès-Nancy, France}
\email{konstantin.usevich@univ-lorraine.fr}

\date{\today}

\keywords{independent component analysis, Jacobi-type algorithms, higher-order cumulant tensors, robust approximate diagonalization, nonsmooth Riemannian optimization, global convergence}

\subjclass[2020]{15A69, 49J52, 65K10, 90C26, 90C30, 94A12}

\begin{abstract}
The sums of the absolute values of higher-order cumulants are classical contrast functions for independent component analysis (ICA) and lead naturally to robust approximate diagonalization of higher-order cumulant tensors using $\ell_1$-based diagonalization over the orthogonal or unitary group. 
We propose Jacobi-SubG, a Jacobi-type framework that selects index pairs
using Riemannian subdifferential information and performs plane rotations
with a refinement step that preserves nonzero diagonal contractions.
For the general tensor model, we prove global convergence of the entire
iterate sequence from any initialization with nonzero diagonal contractions,
provided that the subproblem solvers satisfy
a uniform sufficient-ascent estimate.
For real and complex third- and fourth-order cumulant tensors,
which are widely used in ICA, we develop specialized subproblem
solvers satisfying these conditions.
The resulting algorithms thus enjoy global convergence under
the simple initialization condition that all diagonal contractions
are nonzero.
Numerical experiments on ICA and tensor diagonalization
show that the proposed Jacobi-SubG variants achieve better source recovery
and higher final objective values in several settings, particularly
at the higher tested noise levels.
\end{abstract}

\maketitle

\section{Introduction}

Let $\vect{z}(t)\in \CC^{n}$ be a complex-valued random vector. 
For $d\geq 3$, define the $d$th-order \emph{cumulant tensor}
$\tens{A}\in \CC^{n\times \cdots \times n}$ by
$\tenselem{A}_{i_1\cdots i_d}
=
\operatorname{cum}\bigl(
z_{i_1}^{*},\ldots,z_{i_r}^{*},
z_{i_{r+1}},\ldots,z_{i_d}
\bigr)$ with $1\leq i_1,\ldots,i_d\leq n,$ 
where an integer $r$ satisfies $0\leq r\leq d$ and  $(\cdot)^*$ denotes the complex conjugation.
In this paper, we study the following problem of \emph{robust approximate diagonalization of higher-order cumulant tensors}
\cite[Eq.~(7)]{moreau1996criteria}
\cite[Prop.~3.6 and p.~296]{Como10:book}:
\begin{align}\label{cost-fn-complex-l1-diag}
\max f(\matr{U})
\eqdef
\|\diag{\tens{W}}\|_1
=
\sum_{p=1}^{n} |\tens{A}
\contr{1} \vect{u}_p^{\H}
\cdots
\contr{r} \vect{u}_p^{\H}
\contr{r+1} \vect{u}_p^{\T}
\cdots
\contr{d} \vect{u}_p^{\T}|,
\end{align}
where $\matr{U}= (\vect{u}_1,\dots,\vect{u}_n) \in \UN{n}$ is a unitary matrix and the transformed tensor $\tens{W}$ is defined as 
$\tens{W}
\eqdef 
\tens{A}
\contr{1} \matr{U}^{\H}
\cdots
\contr{r} \matr{U}^{\H}
\contr{r+1} \matr{U}^{\T}
\cdots
\contr{d} \matr{U}^{\T}.$ 
For later use, denote the \(p\)th diagonal contraction by
\(g_p(\matr{U})\eqdef\tenselem{W}_{p\cdots p}\), so that
\(f(\matr{U})=\sum_{p=1}^n|g_p(\matr{U})|\).

The model \eqref{cost-fn-complex-l1-diag} encompasses several criteria widely used in \emph{blind source separation} (BSS) and
\emph{independent component analysis} (ICA) \cite{Como10:book}.
The classical \emph{linear instantaneous mixing} model
\cite[Eq.~(5.18)]{Como10:book} of ICA is
\begin{equation}\label{eq:ICA_Classical}
\vect{y}(t)=\matr{A}\vect{s}(t)+\vect{n}(t),
\end{equation}
where $\vect{y}(t)\in\CC^n$ is the observed mixture,
$\matr{A}$ is the unknown mixing matrix,
$\vect{s}(t)$ is the vector of latent sources, and
$\vect{n}(t)$ is additive noise. 
Let $\vect{z}(t)$ denote the observations after pre-whitening
\cite[Sec.~5.2.2]{Como10:book}.
ICA commonly recovers sources by maximizing 
non-Gaussianity based on higher-order cumulants
\cite{Hyvarinen2001}.
The source estimates are $\vect{x}(t)=\matr{U}^{\T}\vect{z}(t)$.
In particular, fourth- and third-order cumulants give rise to
the following kurtosis- and skewness-based criteria, respectively.

\begin{itemize}
\item[(i)]
Let $\tens{A}$ be the fourth-order cumulant tensor of $\vect{z}(t)$, with
$\tenselem{A}_{ijkl}=\mathrm{cum}(z_i^*,z_j^*,z_k,z_l)$. Then
\eqref{cost-fn-complex-l1-diag} becomes
\begin{align}\label{eq:hermit_sy_order4-new}
f_{ku}(\matr{U})
&=
\left\|
\diag{
\tens{A}
\contr{1}\matr{U}^{\H}
\contr{2}\matr{U}^{\H}
\contr{3}\matr{U}^{\T}
\contr{4}\matr{U}^{\T}}
\right\|_1
=
\sum_{p=1}^{n}
\left|\mathrm{cum}(x_p^*,x_p^*,x_p,x_p)\right|.
\end{align}
This summed absolute-kurtosis objective is a valid contrast 
\cite[Eq.~(42)]{moreau1996high},
\cite[Eq.~(7)]{moreau1996criteria},
\cite[Prop.~3.6, Eq.~(6.14), Eq.~(8.24)]{Como10:book} in ICA.
It is also known as the \emph{multiuser kurtosis} criterion
\cite{papadias2002globally}; its theoretical and algorithmic properties were studied in
\cite{ge2008analysis,ge2010efficient}.

\item[(ii)]
Let $\tens{A}$ be the third-order cumulant tensor, with
$\tenselem{A}_{ijk}=\mathrm{cum}(z_i^*,z_j,z_k)$. The corresponding skewness objective is
\begin{align}\label{eq:hermit_sy_order3-new}
f_{sk}(\matr{U})
&=
\left\|
\diag{
\tens{A}
\contr{1}\matr{U}^{\H}
\contr{2}\matr{U}^{\T}
\contr{3}\matr{U}^{\T}}
\right\|_1
=
\sum_{p=1}^{n}
\left|\mathrm{cum}(x_p^*,x_p,x_p)\right|.
\end{align}
This is also a valid contrast
\cite[Eq.~(42)]{moreau1996high},
\cite[Eq.~(7)]{moreau1996criteria}, 
\cite[Prop.~3.6]{Como10:book} in ICA.
It can be effective for asymmetric sources with weak kurtosis \cite{song2016ecoica}. Weighted combinations of absolute kurtosis and skewness have also been considered
\cite{moreau2002generalization,karvanen2004independent}, but are beyond the scope of this paper.
\end{itemize}

Two classical types of contrast functions are closely related to
\eqref{eq:hermit_sy_order4-new}, \eqref{eq:hermit_sy_order3-new}.
The first sums squared kurtosis or skewness, yielding tensor approximate
diagonalization \cite{comon1994independent,de1997signal}.
The second sums these cumulants directly, yielding approximate
tensor-trace maximization \cite{comon1997improved}
\cite[Sec.~3.2]{comon2007tensor}\cite[Sec.~5.6]{Como10:book},
typically when the source cumulants have a common sign.
Algorithms include JADE \cite{cardoso1993blind,cardoso1996jacobi},
Jacobi-CoM2 \cite{comon1994independent}, and other Jacobi-type methods
\cite{Como10:book,Lathauwer96:simultaneous,begovic2024trace},
with convergence analyses in
\cite{LUC2017globally,ULC2019,li2023convergence}.
A nonsymmetric analogue and an associated Jacobi-type method were also studied
in \cite{martin2008jacobi,sheng2023jacobi}. 
In contrast, the absolute-cumulant objectives \eqref{eq:hermit_sy_order4-new} and
\eqref{eq:hermit_sy_order3-new} reduce the relative dominance of large
cumulants compared with squared-sum criteria and, unlike trace
maximization, require no common-sign assumption.
However, direct optimization is more difficult because these objectives are nonsmooth.
Kurtosis-based FastICA is closely related to maximizing absolute
kurtosis \cite[Sec.~8.2.3]{Hyvarinen2001}
\cite{hyvarinen1997fast,hyvarinen1999fast}, but its fixed-point,
quasi-Newton-type iteration does not directly solve the nonsmooth
subproblems considered here \cite[Sec.~6.10.5]{Como10:book}.
Other methods include gradient-based approaches
\cite[Sec.~4]{moreau1996criteria}\cite[(8.13)]{Hyvarinen2001}
and planar Givens-rotation methods \cite[Sec.~3.3]{moreau1996high}. 
Recently, \cite{li2026ell} introduced an $\ell_1$-based model
for the joint diagonalization of tensors and matrices, solved by
proximal alternating maximization ($\ell_1$-PAM) using a strategy
related to that in~\cite{wang2023linear}.
In addition, \cite{li2026robust} proposed an optimization model based
on the $\ell_1$-norm of the off-diagonal entries of third-order tensors
and developed a Riemannian smoothing conjugate gradient (RSCG)
algorithm to solve it.
Both studies are related to the present work but employ
different algorithms; the off-diagonal model in~\cite{li2026robust}
also differs from ours.

In this paper, we develop Jacobi-type algorithms for
\eqref{cost-fn-complex-l1-diag} and study their global convergence\footnote{In this paper, \emph{global convergence} means convergence of the entire sequence of iterates.}. 
The main
contributions are summarized as follows.
\begin{itemize}
\item
We extend the gradient-based Jacobi-G framework
\cite{IshtAV13:simax,LUC2017globally,ULC2019} to the nonsmooth setting in \eqref{cost-fn-complex-l1-diag}.
The resulting Jacobi-SubG framework uses Riemannian subdifferential
information for index-pair selection and incorporates a refinement
step that preserves nonzero diagonal contractions.
For subproblem solvers satisfying a uniform
sufficient-ascent estimate, we prove global convergence of the entire
iterate sequence for \eqref{cost-fn-complex-l1-diag} from any
initialization with nonzero diagonal contractions.

\item
For \eqref{eq:hermit_sy_order4-new} and
\eqref{eq:hermit_sy_order3-new}, the nonsmooth Jacobi subproblems
have more complex structures than those in the smooth setting
\cite{LUC2017globally,ULC2019}.
We construct specialized subproblem
solvers that seek large objective
increases while satisfying the required uniform sufficient-ascent
estimate. 
These estimates establish
global convergence of the resulting algorithms under
the simple initialization condition that all diagonal contractions
are nonzero.

\item
Numerical experiments evaluate the proposed Jacobi-SubG variants on real
fourth-order ICA and tensor approximate diagonalization problems.
In ICA, they achieve lower mean Amari errors than Jacobi-CoM2
\cite{comon1994independent,LUC2017globally}, FastICA
\cite{hyvarinen1997fast,hyvarinen1999fast}, and RobustICA
\cite{zarzoso2009robust} at every tested setting.
For tensor approximate diagonalization, they attain higher mean final
objectives than the Riemannian subgradient (RSG) baseline adapted from
\cite{li2021weakly} and $\ell_1$-PAM \cite{li2026ell} throughout the noise
sweep and for most tested dimensions. They also have lower median solver
times than RSG at every tested tensor setting.
\end{itemize}

\begin{remark}
Although \eqref{cost-fn-complex-l1-diag} is stated for a complex tensor
$\tens{A}$ and a unitary variable $\matr{U}$, it includes the real-valued
setting as a special case. Restriction to $\RR$ gives
\begin{align}\label{cost-fn-complex-l1-diag-real}
\max_{\matr{Q}\in\ON{n}} f(\matr{Q})
\eqdef
\|\diag{\tens{W}}\|_1
=
\sum_{p=1}^{n}
\left|
\tens{A}
\contr{1}\vect{q}_p^\T
\cdots
\contr{d}\vect{q}_p^\T
\right|,
\end{align}
where $\matr{Q}=(\vect{q}_1,\ldots,\vect{q}_n)$ is orthogonal and
$\tens{W}=\tens{W}(\matr{Q})
\eqdef
\tens{A}
\contr{1}\matr{Q}^\T
\cdots
\contr{d}\matr{Q}^\T.$ 
This model uses real-valued ICA contrast functions based on the absolute
values of higher-order statistics, such as skewness and kurtosis. The
proposed algorithms and convergence results apply to this setting; 
any necessary distinctions are stated below.
\end{remark}

The remainder of this paper is organized as follows.
In \cref{sec:preliminaries}, we review the necessary background on matrix manifolds, nonsmooth analysis, and the {\L}ojasiewicz inequality.
In \cref{sec:jacobi-subg}, we introduce the nonsmooth Jacobi-SubG algorithm, establish the general conditional result and derive an initialization-based convergence theorem for the ICA objectives; the corresponding subproblem solvers are developed in \cref{sec:4th-order-sub,sec:3rd-order-sub-complex-real}.
Finally, numerical results and concluding remarks are presented in \cref{sec:numerical,sec:conclusion}, respectively.

\section{Preliminaries}\label{sec:preliminaries}

\subsection{Riemannian subdifferential}
Endow $\CC^{n\times n}$ with the inner product
$\langle \matr{A},\matr{B}\rangle=\operatorname{Re}\operatorname{tr}(\matr{A}^{\H}\matr{B})$
and view $\UN{n}\subseteq \CC^{n \times n}$ as a smooth embedded submanifold with the induced metric.
We write $\partial f$ for the limiting subdifferential
\cite{rockafellar1998variational} and $\partial_{\mathcal R}f$ for the limiting
Riemannian subdifferential of $f|_{\mathcal M}$ on such a submanifold $\mathcal M$
\cite{ledyaev2007nonsmooth}.
A function $f:\Omega\to\RR$ on an open convex set
$\Omega$ is $\rho$-weakly convex if $f+\frac{\rho}{2}\|\cdot\|_F^2$ is
convex, where $\rho\geq0$.
For such $f$ and $\mathcal M\subseteq\Omega$, we have
\cite{ledyaev2007nonsmooth,li2021weakly}
\begin{equation}\label{eq:riemannian-subdiff}
\partial_{\mathcal R}f(\matr{X})
=\mathcal P_{T_{\matr{X}}\mathcal M}\bigl(\partial f(\matr{X})\bigr),
\qquad \matr{X}\in\mathcal M,
\end{equation}
where $\mathcal P_{T_{\matr{X}}\mathcal M}$ is the orthogonal tangent projection.
For $\UN{n}$, we have
$T_{\matr{U}}\UN{n}=\{\matr{U}\matr{K}:\matr{K}^{\H}=-\matr{K}\}$ and
$\mathcal P_{T_{\matr{U}}\UN{n}}(\matr{Z})=\matr{Z}-\matr{U}\operatorname{sym}(\matr{U}^{\H}\matr{Z})$,
where $\operatorname{sym}(\matr{B})\eqdef (\matr{B}+\matr{B}^{\H})/2$ \cite{boumal2023intromanifolds,li2019polar}.

For our objective \eqref{cost-fn-complex-l1-diag},
$f=\ell\circ\vect{g}$, where $\ell(\vect{z})=\sum_{p=1}^n|z_p|$
and $\vect{g}=(g_1,\ldots,g_n)^{\T}$.
Since $\ell$ is convex and Lipschitz and each $g_p$ is polynomial,
$f$ is weakly convex on any bounded open convex neighborhood of $\UN{n}$
\cite{drusvyatskiy2019efficiency}. Define
$\set{S}(z)=\{z/|z|\}$ for $z\ne0$ and
$\set{S}(0)=\{\alpha\in\CC:|\alpha|\leq1\}$.
Let $Dg_p(\matr{U})^*$ denote the adjoint of the real differential.
The chain rule \cite{rockafellar1998variational} gives
\begin{equation}\label{eq:tensor-subdiff}
\partial f(\matr{U})
=\left\{
\sum_{p=1}^nDg_p(\matr{U})^*[\alpha_p]:
\alpha_p\in\set{S}(g_p(\matr{U})),\quad p=1,\ldots,n
\right\}.
\end{equation}

By \eqref{eq:riemannian-subdiff} and \eqref{eq:tensor-subdiff},
one possible Riemannian subgradient is obtained by setting
$\widehat\alpha_p=g_p(\matr{U})/|g_p(\matr{U})|$
when $g_p(\matr{U})\ne0$ and $\widehat\alpha_p=0$ otherwise:
\begin{equation}\label{eq:selected-riemannian-subgradient}
\matr{G}_{\mathrm R}(\matr{U})
=\mathcal P_{T_{\matr{U}}\UN{n}}
\left(\sum_{p=1}^nDg_p(\matr{U})^*[\widehat\alpha_p]\right)
\in\partial_{\mathcal R}f(\matr{U}).
\end{equation}

\subsection{The {\L}ojasiewicz gradient inequality}

The classical {\L}ojasiewicz gradient inequality holds for real-analytic
functions \cite{lojasiewicz1963propriete,law1965ensembles} and has a
nonsmooth extension \cite{bolte2007lojasiewicz}. 
For further details and subsequent developments, 
see \cite{attouch2013convergence,bolte2014proximal}.

\begin{definition}\label{def:Nonsmooth_KL}
Let $\phi:\RR^n\to\RR\cup \{+\infty\}$ be proper and lower semicontinuous.
The function $\phi$ is said to satisfy a
\emph{nonsmooth {\L}ojasiewicz inequality}
at a point $\vect{x}^\ast\in\operatorname{dom} \partial\phi$
if there exist constants $c>0$, $\mu \in [0,1)$ and a neighborhood $\set{V}$ of $\vect{x}^\ast$
such that for all $\vect{x}\in \set{V}\cap\operatorname{dom}\phi$ with
$\phi(\vect{x})\ne\phi(\vect{x}^\ast)$, we have
\begin{equation}
\label{eq:nonsmooth_loja}
\dist\bigl(\vect{0},\partial \phi(\vect{x})\bigr)
\;\ge\;
c\,|\phi(\vect{x})-\phi(\vect{x}^\ast)|^{\mu}.
\end{equation}
\end{definition}

The following general discrete convergence result can be obtained using the standard finite-length
argument underlying \cite[Theorem~3.2]{AbsMA05:sjo}; see also the continuous
{\L}ojasiewicz convergence theorem
\cite{lojasiewicz1982trajectoires,lageman2002konvergenz}. 

\begin{lemma}\label{lem:euclidean-KL-sequence}
Let $\phi:\RR^n\to\RR\cup \{+\infty\}$ be continuous on $\operatorname{dom} \phi$. 
Suppose that a sequence $\{\vect{x}_k\}_{k\ge 0}\subseteq \operatorname{dom} \phi$ has an accumulation point $\vect{x}_\ast$. 
Assume that the following conditions hold: 
\begin{enumerate}
\item[\textup{(A1)}] there exist $\sigma>0, \alpha\geq 0, \beta\geq 0$ such that, for all large enough $k$, 
\begin{equation*}
\phi(\vect{x}_k)-\phi(\vect{x}_{k+1}) \geq\sigma\, \dist\bigl(\vect{0},\partial \phi(\vect{x}_k)\bigr)^\alpha\,\|\vect{x}_{k+1}-\vect{x}_k\|^\beta;
\end{equation*}
\item[\textup{(A2)}] there exist
$\kappa>0, \eta\geq0$ 
such that, for all large enough $k$ satisfying $\vect{x}_{k+1}\neq\vect{x}_k$, 
\begin{equation*}
\|\vect{x}_{k+1}-\vect{x}_k\|^{\beta-1}
\geq
\kappa\,\dist\bigl(\vect{0},\partial \phi(\vect{x}_k)\bigr)^{\eta};
\end{equation*} 
\item[\textup{(A3)}] for all sufficiently large $k$, it holds that  
\begin{equation*}
\phi(\vect{x}_{k+1}) = \phi(\vect{x}_k)\ \Longrightarrow\ \vect{x}_{k+1}=\vect{x}_k.
\end{equation*}
\end{enumerate}
If $\phi$ satisfies the nonsmooth \L{}ojasiewicz inequality \eqref{eq:nonsmooth_loja} at $\vect{x}_\ast$, and 
$(\alpha+\eta)\mu<1$, 
where $\mu$ is from \eqref{eq:nonsmooth_loja}, 
then the sequence $\{\vect{x}_k\}_{k \geq 0}$ converges to $\vect{x}_\ast$.
\end{lemma}

\begin{proof}
Let $\ell=\phi(\vect{x}_\ast)$ and choose a subsequence
$\vect{x}_{k_j}\to\vect{x}_\ast$. The assumed nonsmooth
{\L}ojasiewicz inequality ensures that
$\vect{x}_\ast\in\operatorname{dom}\phi$, so continuity gives
$\phi(\vect{x}_{k_j})\to\ell$. By \textup{(A1)}, the objective values
are eventually nonincreasing and hence converge to $\ell$.
Thus $e_k:=\phi(\vect{x}_k)-\ell\geq0$ for all sufficiently large $k$.
If $e_k=0$ at such an index, monotonicity and \textup{(A3)} make the
sequence constant thereafter, with value $\vect{x}_\ast$.
We may therefore assume that $e_k>0$ throughout the tail.

Choose $r>0$ and $c>0$ such that the nonsmooth
{\L}ojasiewicz inequality holds on
$B_r(\vect{x}_\ast)\cap\operatorname{dom}\phi$.
Set $d_k=\dist(\vect{0},\partial\phi(\vect{x}_k))$,
$s_k=\|\vect{x}_{k+1}-\vect{x}_k\|$, and
$p=1-(\alpha+\eta)\mu\in(0,1]$.
For all sufficiently large $k$ with
$\vect{x}_k\in B_r(\vect{x}_\ast)$ and $s_k>0$,
\textup{(A1)}, \textup{(A2)}, and the nonsmooth
{\L}ojasiewicz inequality yield
\[
e_k-e_{k+1}
\geq \sigma\kappa d_k^{\alpha+\eta}s_k
\geq \sigma\kappa c^{\alpha+\eta}e_k^{1-p}s_k.
\]
By concavity,
$e_k^p-e_{k+1}^p\geq p e_k^{p-1}(e_k-e_{k+1})$, and therefore
\begin{equation}\label{eq:finite_step}
s_k\leq C(e_k^p-e_{k+1}^p),
\qquad C=(\sigma\kappa p c^{\alpha+\eta})^{-1}.
\end{equation}
This estimate also holds trivially when $s_k=0$.

Choose $k_0$ from the convergent subsequence, sufficiently large that
the preceding estimates apply and
$\|\vect{x}_{k_0}-\vect{x}_\ast\|+Ce_{k_0}^p<r$.
If $j>k_0$ were the first index with
$\vect{x}_j\notin B_r(\vect{x}_\ast)$, then
\eqref{eq:finite_step} would hold for every $k_0\leq k<j$,
including the step leaving the ball. Summing gives
\[
\|\vect{x}_j-\vect{x}_\ast\|
\leq \|\vect{x}_{k_0}-\vect{x}_\ast\|
+\sum_{k=k_0}^{j-1}s_k
\leq \|\vect{x}_{k_0}-\vect{x}_\ast\|+Ce_{k_0}^p<r,
\]
a contradiction. Hence the entire tail remains in the ball, and
\eqref{eq:finite_step} implies
$\sum_{k=k_0}^{\infty}s_k\leq Ce_{k_0}^p<\infty$.
Thus the sequence is Cauchy and, since a subsequence converges to
$\vect{x}_\ast$, the whole sequence converges to $\vect{x}_\ast$.
\end{proof}

\begin{remark}
When $\alpha=\beta=1$ and $\eta=0$, condition~\textup{(A2)}
holds automatically, so the convergence criterion essentially reduces
to \textup{(A1)} and \textup{(A3)}, in line with the smooth framework
of \cite[Theorem~3.2]{AbsMA05:sjo}.
Another common choice is $\alpha=0$, $\beta=2$, and $\eta=1$,
which gives quadratic sufficient decrease and a relative-error bound
of the type used in \cite[Sec.~2.3]{attouch2013convergence}.
Beyond these two standard cases, the general framework of
\cref{lem:euclidean-KL-sequence} provides more choices of
exponents that guarantee convergence under the stated assumptions.
\end{remark}

We next state the intrinsic Riemannian counterpart. It extends the smooth
formulation in \cite[Definition~2.1]{SU15:pro}; see also
\cite{huang2022riemannian} for Riemannian
Kurdyka--{\L}ojasiewicz inequalities.

\begin{definition}
Let $\varphi:\RR^n\to\RR$ be lower semicontinuous. 
Let $\set{M}\subseteq\RR^{n}$ be a finite-dimensional smooth Riemannian submanifold.
We say that $\varphi$ satisfies the \emph{Riemannian nonsmooth {\L}ojasiewicz inequality}
at a point $\vect{x}^\ast\in\operatorname{dom}\partial_{\mathcal R}\varphi\subseteq\set{M}$
if there exist constants $c>0$, $\mu \in [0,1)$ and a neighborhood $\set{V}$ of $\vect{x}^\ast$ in $\set{M}$
such that for all $\vect{x}\in \set{V}$ with
$\varphi(\vect{x})\ne\varphi(\vect{x}^\ast)$, it holds that
\begin{equation}\label{eq:riemannian-lojasiewicz}
\dist\bigl(\vect{0},\partial_{\mathcal R}\varphi(\vect{x})\bigr)
\;\ge\;
c\,\bigl|\varphi(\vect{x})-\varphi(\vect{x}^\ast)\bigr|^\mu.
\end{equation}
\end{definition}

We now present a Riemannian version of
\cref{lem:euclidean-KL-sequence}, providing a nonsmooth
counterpart of \cite[Theorem~2.3]{SU15:pro} in the
smooth-manifold setting (when $\alpha=\beta=1$ and $\eta=0$). The proof follows by applying
\cref{lem:euclidean-KL-sequence} to
$\phi:=\varphi+\delta_{\set{M}}$, where $\delta_{\set{M}}$
is the indicator function of $\set{M}$.

\begin{proposition}
\label{prop:riemannian-KL-convergence}
Let $\varphi:\RR^n\to\RR$ be continuous, and let
$\set{M}\subseteq\RR^n$ be a closed smooth embedded submanifold.
Suppose that a sequence
$\{\vect{x}_k\}_{k\geq0}\subseteq\set{M}$
has an accumulation point $\vect{x}_\ast\in\set{M}$.
Assume that the following conditions hold:
\begin{enumerate}
\item[\textup{(A1)}]
there exist $\sigma>0, \alpha\geq 0, \beta\geq 0$ such that, for all large enough $k$,  
\[
\varphi(\vect{x}_k)-\varphi(\vect{x}_{k+1})
\geq
\sigma\,\dist\bigl(\vect{0},
\partial_{\mathcal R}\varphi(\vect{x}_k)\bigr)^\alpha
\|\vect{x}_{k+1}-\vect{x}_k\|^\beta.
\]
\item[\textup{(A2)}]
there exist
$\kappa>0, \eta\geq0$ 
such that, for all large enough $k$ satisfying $\vect{x}_{k+1}\neq\vect{x}_k$,
\[
\|\vect{x}_{k+1}-\vect{x}_k\|^{\beta-1}
\geq
\kappa\,\dist\bigl(\vect{0},
\partial_{\mathcal R}\varphi(\vect{x}_k)\bigr)^\eta.
\]
\item[\textup{(A3)}]
for all sufficiently large $k$, it holds that  
\[
\varphi(\vect{x}_{k+1})=\varphi(\vect{x}_k)
\quad\Longrightarrow\quad
\vect{x}_{k+1}=\vect{x}_k.
\]
\end{enumerate}
If $\varphi$ satisfies the Riemannian nonsmooth
{\L}ojasiewicz inequality
\eqref{eq:riemannian-lojasiewicz} at $\vect{x}_\ast$
with exponent $\mu$ satisfying $(\alpha+\eta)\mu<1$,
then $\vect{x}_k\to\vect{x}_\ast$.
\end{proposition}

\section{Subgradient-based Jacobi-type Algorithms} \label{sec:jacobi-subg}

In this section, we introduce the Jacobi-SubG framework and establish
its global convergence. 
The proofs of the results in this section are postponed to \cref{app:jacobi-subg-proofs}.

\subsection{Jacobi-SubG algorithm} \label{subproblem-general-2} 
Let $(i,j)$ be an index pair with $1 \leq i < j \leq n$.
Then, for a matrix $\matr{\Psi} \in \UN{2}$, we define the \emph{plane transformation} in $\UN{n}$ as
\begin{equation}\label{eq:2.8}
\matr{G}^{(i,j,\matr{\Psi})}
=
\left[
\begin{array}{ccccccc}
1 &        &        &        &        &        &   \\
& \ddots &        &        &        &        &   \\
&        & \Psi_{1,1} & \cdots & \Psi_{1,2} &        &   \\
&        & \vdots & \ddots & \vdots &        &   \\
&        & \Psi_{2,1} & \cdots & \Psi_{2,2} &        &   \\
&        &        &        &        & \ddots &   \\
&        &        &        &        &        & 1
\end{array}
\right],
\end{equation}
which coincides with $\matr{I}_n$ except at the positions $(i,i)$, $(i,j)$, $(j,i)$, and $(j,j)$.
Let $f$ be the objective function in \eqref{cost-fn-complex-l1-diag}, and fix
$\matr{U}\in\UN{n}$.
We define the following elementary function:
\begin{equation}\label{eq:2.9}
h_{(i,j),\matr{U}}:\UN{2} \to \RR,
\ \ 
\matr{\Psi} \mapsto f\bigl(\matr{U} \matr{G}^{(i,j,\matr{\Psi})}\bigr).
\end{equation}
To solve the subproblem in \eqref{eq:2.9}, we need to parameterize $\matr{\Psi}$.
In this paper, we only consider \emph{rotations}:
\begin{equation}\label{eq:2.10}
\matr{\Psi}(c,s_1,s_2)
=
\begin{bmatrix}
c & -s \\
s^* & c
\end{bmatrix}
=
\begin{bmatrix}
c & -(s_1 + i s_2) \\
s_1 - i s_2 & c
\end{bmatrix}
=
\begin{bmatrix}
\cos\theta & -\sin\theta e^{i\varphi} \\
\sin\theta e^{-i\varphi} & \cos\theta
\end{bmatrix},
\end{equation}
where $c \geq0$, $s = s_1 + i s_2 \in \CC$, $c^2 + |s|^2 = 1$.
It is worth noting that the set of matrices in \eqref{eq:2.10} is a proper subset of that in \eqref{eq:2.8}. 
This restriction is justified because, as pointed out in \cite[Eq. (2.11)]{ULC2019}, 
the objective function \eqref{cost-fn-complex-l1-diag} is invariant under multiplication of the columns of $\matr{U}$ by scalars from $\TT\eqdef \{z \in \mathbb{C} : |z| = 1\}$, i.e.,
\begin{equation*}
f(\matr{U}) = f(\matr{U}\matr{S})
\ \ 
\text{for all }
\matr{S} =
\begin{bmatrix}
z_1 &        & 0 \\
& \ddots &   \\
0   &        & z_n
\end{bmatrix},
\quad
z_i \in \TT,\ 1 \leq i \leq n.
\end{equation*}
As in \cite{ULC2019}, we refer to $\matr{G}^{(i,j,\matr{\Psi})}$ with $\matr{\Psi}$ of the form
\eqref{eq:2.10} as \emph{Givens rotations}, and to maximizers of $h_{(i,j),\matr{U}}$ as
\emph{Jacobi rotations}. 

\begin{remark}
In the real-valued setting, when solving the optimization model
\eqref{cost-fn-complex-l1-diag-real}, we need to employ the following
\emph{Givens rotations}:
\begin{equation}\label{eq:Given_rotations}
\matr{\Psi}(c,s)
=
\begin{bmatrix}
c & -s \\
s & c
\end{bmatrix}
=
\begin{bmatrix}
\cos\theta & -\sin\theta  \\
\sin\theta & \cos\theta
\end{bmatrix},
\end{equation}
where $c \geq0$ and $c^2 + s^2 = 1$.
\end{remark}

As a nonsmooth variant of the Jacobi-G algorithm in \cite{IshtAV13:simax,LUC2017globally,ULC2019}, we propose the following \emph{Subgradient-based Jacobi-type} (Jacobi-SubG) algorithm to maximize $f$. 
Compared with the smooth Jacobi-G framework, our algorithm includes
an additional refinement step using
Subalgorithm~\ref{alg:nonvanishing-safeguard}.
This safeguard enables us to establish global convergence under
a simpler initialization condition; see
\cref{subsec:conver_gence_general}. 
At iteration $k$, write
$h_k\eqdef h_{(i_k,j_k),\matr{U}_{k-1}}$.
For $s_1^2+s_2^2<1$, define
\[
\widetilde h_k(s_1,s_2)
\eqdef h_k\bigl(\matr{\Psi}(\sqrt{1-s_1^2-s_2^2},s_1,s_2)\bigr).
\]
We write $\partial h_k(\matr{I}_2)$ for
$\partial\widetilde h_k(0,0)$.
In the real case, define
$\widetilde h_k(s)\eqdef h_k(\matr{\Psi}(\sqrt{1-s^2},s))$
for $|s|<1$ and write $\partial h_k(\matr{I}_2)$ for
$\partial\widetilde h_k(0)$.
We use the Euclidean norm for these subgradient vectors.

\begin{algorithm}
\caption{Jacobi-SubG Algorithm}\label{Jacobi-SubG-general}
\begin{algorithmic}[1]
\STATE \textbf{Input:} $f$ in \eqref{cost-fn-complex-l1-diag},
$0<\delta\leq 2/\sqrt{n(n-1)}$, and
$\matr{U}_0\in\UN{n}$ with $g_p(\matr{U}_0)\neq0$ for all $p$.
\STATE \textbf{Output:} iterates $\{\matr{U}_k\}$.
\FOR{$k=1,2,\ldots$}
\STATE \textbf{Step I:}
Choose $(i_k,j_k)$ such that there exists
  $\vect{\xi}\in
  \partial h_k(\matr{I}_2)$ satisfying
\begin{equation}\label{inequality-gra-based-2}
\|\vect{\xi}\|\geq
\delta\dist(\vect{0},\partial_{\mathcal R}f(\matr{U}_{k-1})).
\end{equation}
\vspace{-\baselineskip}
\STATE \textbf{Step II:} Generate $\widehat{\matr{\Psi}}_k$ using Subalgorithm CG (Candidate Generation). 
\STATE \textbf{Step III:}  Refine $\widehat{\matr{\Psi}}_k$ using
Subalgorithm~\ref{alg:nonvanishing-safeguard} (Candidate Refinement) to obtain $\matr{\Psi}_k$.
\STATE \textbf{Step IV:}  Update
$\matr{U}_k=\matr{U}_{k-1}\Gmat{i_k}{j_k}{\matr{\Psi}_k}$.
\ENDFOR
\end{algorithmic}
\end{algorithm}

\begingroup
\floatname{algorithm}{Subalgorithm}
\renewcommand{\thealgorithm}{CR}
\begin{algorithm}
\caption{Candidate Refinement}
\label{alg:nonvanishing-safeguard}
\begin{algorithmic}[1]
\STATE \textbf{Input:} $\matr{U}_{k-1}$ with $g_p(\matr{U}_{k-1})\neq0$
for all $p$, $(i_k,j_k)$, $h_k$, and a candidate rotation
$\widehat{\matr{\Psi}}_k$ with
$h_k(\widehat{\matr{\Psi}}_k)\geq h_k(\matr{I}_2)$.
\STATE \textbf{Output:} an accepted rotation $\matr{\Psi}_k$.
\STATE Set
$\widehat{\Delta}_k
=h_k(\widehat{\matr{\Psi}}_k)-h_k(\matr{I}_2)$ and
$\widehat{\matr{U}}_k
=\matr{U}_{k-1}\Gmat{i_k}{j_k}{\widehat{\matr{\Psi}}_k}$.
\IF{$\widehat{\Delta}_k=0$}
\STATE Return $\matr{I}_2$.
\ELSIF{$g_{i_k}(\widehat{\matr{U}}_k)
g_{j_k}(\widehat{\matr{U}}_k)\neq0$}
\STATE Return $\widehat{\matr{\Psi}}_k$.
\ENDIF
\STATE Let $\widehat\theta_k\in[0,\pi/2]$ be the angle of
$\widehat{\matr{\Psi}}_k$ in \eqref{eq:2.10}.
Let $\matr{\Psi}_k(t)$ have angle $t\widehat\theta_k$ and the same phase, and set
$\matr{U}_k(t)
=\matr{U}_{k-1}\Gmat{i_k}{j_k}{\matr{\Psi}_k(t)}$.
\STATE Choose the smallest $m\geq1$ such that $t_m=1-2^{-m}$ satisfies
\[
\begin{aligned}
g_{i_k}(\matr{U}_k(t_m))g_{j_k}(\matr{U}_k(t_m))\neq0,\ \
h_k(\matr{\Psi}_k(t_m))-h_k(\matr{I}_2)
\geq\tfrac12\widehat{\Delta}_k.
\end{aligned}
\]
\STATE Return $\matr{\Psi}_k(t_m)$.
\end{algorithmic}
\end{algorithm}
\endgroup

\begin{proposition}
\label{prop:admissible-pair}
In \cref{Jacobi-SubG-general}, for every
$\matr{U}_{k-1}\in\UN{n}$, there exist an index pair
$(i_k,j_k)$ and
$\vect{\xi}\in
\partial h_k(\matr{I}_2)$
satisfying \eqref{inequality-gra-based-2}.
\end{proposition}

\subsection{Convergence analysis}\label{subsec:conver_gence_general}

We first record the properties of
the safeguard; finite termination follows from analyticity along the
rotation path and continuity of the objective.

\begin{lemma}
\label{lem:nonvanishing-safeguard}
Suppose $g_p(\matr{U}_{k-1})\neq0$ for all $p$ and
$\widehat{\Delta}_k\geq0$.
Then Subalgorithm~\ref{alg:nonvanishing-safeguard} terminates after
finitely many trials. Its output satisfies $g_p(\matr{U}_k)\neq0$
for all $p$,
$$h_k(\matr{\Psi}_k)-h_k(\matr{I}_2)\geq
\tfrac12\widehat{\Delta}_k,\ \ \|\matr{\Psi}_k-\matr{I}_2\|_F
\leq\|\widehat{\matr{\Psi}}_k-\matr{I}_2\|_F.$$
Moreover, $h_k(\matr{\Psi}_k)=h_k(\matr{I}_2)$ implies
$\matr{\Psi}_k=\matr{I}_2$.
\end{lemma}

The safeguard preserves a candidate sufficient-ascent estimate up to
a factor of two. The following theorem combines this fact with the
nonsmooth {\L}ojasiewicz inequality.

\begin{theorem}[Global convergence of Jacobi-SubG]
\label{thm:jacobi-conditional}
Consider \cref{Jacobi-SubG-general} applied to
\eqref{cost-fn-complex-l1-diag}, with $g_p(\matr{U}_0)\neq0$
for all $p$.
Suppose that every candidate satisfies,
for all sufficiently large $k$,
\begin{equation}\label{eq:subproblem_sufficient_ascent}
h_k(\widehat{\matr{\Psi}}_k)-h_k(\matr{I}_2)
\geq
\varepsilon_{\mathrm{sub}}\,
\dist\!\left(
\vect{0},\partial h_k(\matr{I}_2)
\right)
\|\widehat{\matr{\Psi}}_k-\matr{I}_2\|_F
\end{equation}
for some constant $\varepsilon_{\mathrm{sub}}>0$ independent of $k$.
Then $g_p(\matr{U}_k)\neq0$ for all $p$ and $k$, and the whole
sequence $\{\matr{U}_k\}_{k\geq0}$ converges to a point
$\matr{U}_*\in\UN{n}$.
\end{theorem}

For the fourth- and third-order objectives, Step II of
\cref{Jacobi-SubG-general} uses the corresponding subalgorithms for candidate generation in
\cref{sec:4th-order-sub,sec:3rd-order-sub-complex-real}.
The candidate estimates in
\cref{prop:uniform-ascent-complex4,prop:uniform-ascent-complex3},
respectively, ensure the required sufficient ascent, yielding the
following direct consequence.

\begin{corollary}[Convergence for the cumulant objectives]
\label{thm:jacobi-specialized}
Consider \cref{Jacobi-SubG-general} applied to either
\eqref{eq:hermit_sy_order4-new} or
\eqref{eq:hermit_sy_order3-new}.
Suppose that $g_p(\matr{U}_0)\neq0$ for all $p$ and that every
candidate rotation is generated by the corresponding subalgorithm in
\cref{sec:4th-order-sub,sec:3rd-order-sub-complex-real}.
Then the whole sequence $\{\matr{U}_k\}_{k\geq0}$ converges to
a point $\matr{U}_*\in\UN{n}$.
\end{corollary}

\begin{remark}
Although stated for the complex case, all results in this section
remain valid in the real case \eqref{cost-fn-complex-l1-diag-real} under the corresponding assumptions,
with $\UN{n}$ replaced by $\ON{n}$, conjugation omitted, and the
corresponding real subproblem solvers used.
\end{remark}

\section{Solving the subproblems for the fourth-order cumulant tensor case}\label{sec:4th-order-sub}

For an index pair $(i,j)$, let $\tens{W}(\matr{U})^{(i,j)}$
denote the subtensor obtained by restricting every mode to
$\{i,j\}$, with local indices $1$ and $2$ corresponding to
$i$ and $j$, respectively.
For
\(\mathsf X\in\{\mathrm{R4},\mathrm{C4},\mathrm{R3},\mathrm{C3}\}\),
we write \(h_{k,\mathsf X}\) for the corresponding reduced objective
formed from the two selected diagonal entries. Thus
\(h_k=C_{k,\mathsf X}+h_{k,\mathsf X}\), where
\(C_{k,\mathsf X}\) is independent of the Givens parameters.
Consequently, the two functions have the same increments and local
subdifferentials. For functions on the unit sphere \(\Sphere{m}\),
\(\partial_{\Sphere{m}}\) denotes the limiting Riemannian
subdifferential. For matrices, $\|\cdot\|_2$ denotes the spectral norm.
In the subalgorithms below, the starred parameters specify the
candidate $\widehat{\matr{\Psi}}_k$, before refinement by
Subalgorithm~\ref{alg:nonvanishing-safeguard}.
The ascent estimates require only that the two selected contractions
at $\matr{U}_{k-1}$ are nonzero; their constants do not depend on
the magnitudes of these contractions.
By \cref{lem:nonvanishing-safeguard}, this condition holds at every
iteration when $g_p(\matr{U}_0)\neq0$ for all $p$.

\subsection{Fourth-order real case}
For the real fourth-order model \eqref{cost-fn-complex-l1-diag-real},
we first give an explicit representation of the two selected diagonal
contractions.

\begin{lemma}
\label{lem:real4-representation}
Let $\tens{W} \in \RR^{2\times 2\times 2\times 2}$ be symmetric. 
Let $\matr{\Psi}(c,s)\in\ON{2}$ be as in \eqref{eq:Given_rotations}, and define
$\tens{T} = \tens{W} \contr{1} \matr{\Psi}^\T \contr{2} \matr{\Psi}^\T \contr{3} \matr{\Psi}^\T \contr{4} \matr{\Psi}^\T.$
Define $\vect{x} = (c,s)^\T\in \Sphere{1}$ and
\begin{equation}\label{eq:rcs_vect}
\vect{r} = \vect{r}(c,s) \eqdef
[2c^2 - 1, 2c s]^{\T} \in \Sphere{1}.
\end{equation}
Then we have 
$\tenselem{T}_{1111} = \vect{r}(c,s)^\T \matr{\Gamma} \vect{r}(c,s) + \vect{b}^\T \vect{r}(c,s), 
\tenselem{T}_{2222} = \vect{r}(c,s)^\T \matr{\Gamma} \vect{r}(c,s) - \vect{b}^\T \vect{r}(c,s),$ 
where
\begin{align}\label{eq:gamma_matr_b}
\matr{\Gamma} =
\begin{pmatrix}
\frac{\tenselem{W}_{1111} + \tenselem{W}_{2222}}{2}
&
\frac{\tenselem{W}_{1112} - \tenselem{W}_{1222}}{2}
\\
\frac{\tenselem{W}_{1112} - \tenselem{W}_{1222}}{2}
&
\frac{2\tenselem{W}_{1122}+\tenselem{W}_{1111} + \tenselem{W}_{2222}}{4} + \tenselem{W}_{1212}
\end{pmatrix},\ \
\vect{b} =
\begin{pmatrix}
\frac{\tenselem{W}_{1111} - \tenselem{W}_{2222}}{2} \\
\tenselem{W}_{1112} + \tenselem{W}_{1222}
\end{pmatrix}.
\end{align}
\end{lemma}

The reduced objective is therefore
\begin{equation}\label{eq:subpro_hk_real}
h_{k,\mathrm{R4}}(c,s)=h_{k,\mathrm{R4}}(\matr{\Psi})
=2\max\{|\vect{r}(c,s)^\T\matr{\Gamma}\vect{r}(c,s)|,\,
|\vect{b}^\T\vect{r}(c,s)|\},
\end{equation}
because $|a+b|+|a-b|=2\max\{|a|,|b|\}$ for real $a,b$. Here $\matr{\Gamma}$ and $\vect{b}$ are constructed from $\tens{W}(\matr{U}_{k-1})^{(i_k,j_k)}$ by \eqref{eq:gamma_matr_b}. Let $\lambda_1\geq\lambda_2$ be the eigenvalues of $\matr{\Gamma}$, with an orthonormal eigenbasis. The two branches are
\begin{equation}\label{eq:q_func_cs}
q(c,s)=\vect{r}(c,s)^\T\matr{\Gamma}\vect{r}(c,s),\ \ 
\tilde q(c,s)=|q(c,s)|,\ \  \tilde p(c,s)=|\vect{b}^\T\vect{r}(c,s)|.
\end{equation}
Fix $0<\tau<1$ and define
\begin{equation}\label{eq:tildeq_eta_def}
\tilde q_\tau(c,s)=\tilde q(c,s)-\tau(|cs|+s^2).
\end{equation}
Compactness of the orthogonal group and continuity of the transformed tensor
imply a uniform bound \(\|\matr{\Gamma}\|_2\leq M_{\mathrm{R4}}\)
for some fixed \(M_{\mathrm{R4}}>0\).

{
\floatname{algorithm}{Subalgorithm}
\renewcommand{\thealgorithm}{CG-R4}
\begin{algorithm}
\caption{Candidate generation for the reduced objective \eqref{eq:subpro_hk_real}}\label{alg:branch-solver-h}
\begin{algorithmic}[1]
\STATE \textbf{Input:} the elementary function \eqref{eq:subpro_hk_real}, fixed $\upsilon>0$, and $0<\tau<1$.
\STATE \textbf{Output:} a rotation $\matr{\Psi}(c^*,s^*)$.
\IF{$|\Gamma_{11}|\geq|b_1|$}
\STATE Choose the quadratic branch.
\STATE \textbf{if}
\begin{equation}\label{eq:condi_4real}
\lambda_1\lambda_2\geq0,\quad\text{or}\quad
\lambda_1\lambda_2<0\ \text{and}\ 
|\lambda_1+\lambda_2|\geq\upsilon|\Gamma_{12}|,
\end{equation}
\textbf{then} choose a unit eigenvector $\vect{r}^*$ for an eigenvalue of largest absolute value, with $r_1^*\geq0$.
\vspace{-\baselineskip}
\STATE \textbf{else} choose a global maximizer $(c^*,s^*)$ of $\tilde q_\tau$ over $c^2+s^2=1$, $c\geq0$.
\ELSE
\STATE Choose the linear branch and set $\vect{r}^*=\operatorname{sign}(b_1)\vect{b}/\|\vect{b}\|$.
\ENDIF
\STATE Whenever $\vect{r}^*$ was computed, recover $c^*=\sqrt{(1+r_1^*)/2}$ and $s^*=r_2^*/(2c^*)$. Among tied candidates choose one with maximal $c^*$.
\STATE If $h_{k,\mathrm{R4}}(\matr{\Psi}(c^*,s^*))=
h_{k,\mathrm{R4}}(\matr{I}_2)$, set $(c^*,s^*)=(1,0)$.
\STATE Return $\matr{\Psi}(c^*,s^*)$.
\end{algorithmic}
\end{algorithm}
} 

\begin{remark}
In the real fourth-order setting, \eqref{eq:subpro_hk_real} is algebraically
equivalent to the pairwise kurtosis subproblem in
\cite{ge2010efficient}; both yield the same two-branch objective under a
Givens rotation. The method in \cite{ge2010efficient} globally maximizes
both branches and compares their optimal values. In contrast,
\subalgref{alg:branch-solver-h} optimizes the branch active at the
identity, and is followed by a
contraction-preserving refinement. Although this update need not globally
solve every pairwise subproblem, it provides the uniform sufficient ascent
required for the global convergence of \cref{Jacobi-SubG-general}.
\end{remark}

The selected branch is active at the identity, and any penalty is
nonnegative and vanishes there. Hence every candidate is nondecreasing
for the reduced objective.

\begin{proposition}
\label{prop:uniform-ascent-real4}
Let $h_{k,\mathrm{R4}}(c,s)$ be defined as in
\eqref{eq:subpro_hk_real}. Suppose
$g_{i_k}(\matr{U}_{k-1})g_{j_k}(\matr{U}_{k-1})\neq0$.
Then the candidate returned by
\subalgref{alg:branch-solver-h} satisfies
\[
h_{k,\mathrm{R4}}(c^*,s^*)-h_{k,\mathrm{R4}}(1,0)
\geq
\varepsilon_{\mathrm{R4}}\,
\dist\left(\vect{0},
\partial_{\Sphere{1}}h_{k,\mathrm{R4}}(1,0)\right)
\left\|\matr{\Psi}(c^*,s^*)-\matr{I}_2\right\|_F,
\]
where
\(\varepsilon_{\mathrm{R4}}
\eqdef\frac12\min\{\frac12,\frac{\upsilon}{2\sqrt2},
\frac{\tau}{4M_{\mathrm{R4}}}\}>0\) is independent of \(k\).
\end{proposition}

We next give the auxiliary estimates, beginning with a Rayleigh-quotient
bound related to \cite[Lemma 5.9]{LUC2017globally}.

\begin{lemma}
\label{lem:real4-rayleigh-ascent}
Let $q(c,s)$ be defined as in \eqref{eq:q_func_cs}. 
Let \(\vect{r}^{\ast}=\vect{r}(c^\ast,s^\ast)\in\Sphere{1}\) be a unit eigenvector corresponding to the
largest eigenvalue of \(\matr{\Gamma}\). 
We choose the sign of this eigenvector
such that its first component is nonnegative, and then choose
\((c^\ast,s^\ast)\in \Sphere{1}\) satisfying \(c^\ast\ge 0\). 
Then the following estimate holds with $\varepsilon=1/2$,
independently of $\matr{\Gamma}$: 
\begin{equation}\label{eq:real4-rayleigh-ascent}
q(c^*,s^*)-q(1,0)
\geq
\varepsilon 
\|\grad q(1,0)\|
\sqrt{1-c^*}. 
\end{equation}
\end{lemma}

\begin{proof}
Use the parametrization \eqref{eq:rcs_vect}, and write $\vect{r}^*=(x,y)^\T$, $D=\lambda_1-\lambda_2\geq0$. The spectral decomposition gives $q(c^*,s^*)-q(1,0)=Dy^2$ and $\|\grad q(1,0)\|=4D|xy|$. If $D=0$ or $y=0$, the assertion follows immediately. Otherwise $x=2(c^*)^2-1\geq0$, so $c^*\in[1/\sqrt2,1]$ and $|y|=2c^*\sqrt{1+c^*}\sqrt{1-c^*}$. Since $x\leq c^*\leq c^*\sqrt{1+c^*}$, we have $y^2\geq2|xy|\sqrt{1-c^*}$. Thus \eqref{eq:real4-rayleigh-ascent} holds with $\varepsilon=1/2$, independently of $\matr{\Gamma}$.
\end{proof}

For the quadratic branch, the spectral separation condition \eqref{eq:condi_4real} controls a possible change of sign between the initial Rayleigh quotient and the selected extremal eigenvalue.

\begin{lemma}
\label{lem:real4-absolute-rayleigh-ascent}
Let $\tilde{q}(c,s)$ be defined as in \eqref{eq:q_func_cs}. 
Suppose that the condition \eqref{eq:condi_4real} holds.   
Let \(\vect{r}^{\ast}=\vect{r}(c^\ast,s^\ast)\in\Sphere{1}\) be a unit eigenvector of \(\matr{\Gamma}\) corresponding to an eigenvalue
\(\lambda_{\ast}\) of largest absolute value. 
We choose the sign of this eigenvector
such that its first component is nonnegative, and then choose
\((c^\ast,s^\ast)\in \Sphere{1}\) satisfying \(c^\ast\ge 0\). 
Then, we have that 
\begin{equation}\label{eq:real4-absolute-rayleigh-ascent}
\tilde{q}(c^*,s^*)-\tilde{q}(1,0)
\geq
\min\left\{
\frac{1}{2},
\frac{\upsilon}{2\sqrt{2}}
\right\}
\dist\left(\vect{0},\partial_{\Sphere{1}} \tilde{q}(1,0)\right)
\sqrt{1-c^*},
\end{equation}
where $\upsilon$ is as in the condition \eqref{eq:condi_4real}. 
\end{lemma}

\begin{proof}
Put $d_0=\dist(\vect{0},\partial_{\Sphere{1}}\tilde q(1,0))$. Differentiating $q(\cos\theta,\sin\theta)$ at zero gives $4\Gamma_{12}$, hence $d_0\leq4|\Gamma_{12}|$. Equality holds if $\Gamma_{11}\ne0$; if $\Gamma_{11}=0$, the identity is a global minimum of the nonnegative function $|q|$, so $d_0=0$.
If $\Gamma_{11}\lambda_*\geq0$ and $\lambda_*\ne0$, apply \cref{lem:real4-rayleigh-ascent} to $\sigma\matr{\Gamma}$, where $\sigma=\operatorname{sign}(\lambda_*)$. Its largest eigenvalue is $|\lambda_*|$ and its initial Rayleigh quotient is $|\Gamma_{11}|$. This gives a gain at least $\frac12d_0\sqrt{1-c^*}$. The case $\lambda_*=0$ is immediate because $\matr{\Gamma}=0$.
If $\Gamma_{11}\lambda_*<0$, the eigenvalues have opposite signs. The Rayleigh bounds $\lambda_2\leq\Gamma_{11}\leq\lambda_1$ and the choice of $\lambda_*$ imply
$|\lambda_*|-|\Gamma_{11}|\geq|\lambda_1+\lambda_2|\geq\upsilon|\Gamma_{12}|$.
Since $c^*\geq1/\sqrt2$, we have $\sqrt{1-c^*}\leq1/\sqrt2$. The last lower bound is therefore at least $\frac{\upsilon}{2\sqrt2}d_0\sqrt{1-c^*}$. The two cases prove \eqref{eq:real4-absolute-rayleigh-ascent}.
\end{proof}

\begin{lemma}
\label{lem:real4-penalized-ascent}
Let \(\tilde{q}_{\tau}(c,s)\) be defined as in \eqref{eq:tildeq_eta_def}. Then \\
(i) a global maximizer $(c^*,s^*)$ of
\(\tilde{q}_{\tau}\), with $c^*\geq 0$, can be obtained from a finite
algebraic candidate set.\\
(ii) we have that 
\begin{equation}\label{eq:real4-penalized-ascent}
\tilde q(c^*,s^*)-\tilde q(1,0)
\geq
\frac{\tau}{4M_{\mathrm{R4}}}
\dist
\left(
\vect{0},\partial_{\Sphere{1}}\tilde q(1,0)
\right)
\sqrt{1-c^*},
\end{equation}
where $\tau$ is as in \eqref{eq:tildeq_eta_def}. 
\end{lemma}

\begin{proof}
The function is continuous on $\Sphere{1}$ and invariant under $(c,s)\mapsto(-c,-s)$, so it attains a maximum with $c\geq0$. In the parametrization above, the maximization is equivalent to maximizing
$|\vect{r}^\T\matr{\Gamma}\vect{r}|-\frac{\tau}{2}(|r_2|+1-r_1)$
over $\vect{r}\in\Sphere{1}$. To obtain a maximizer algebraically, use $|a|=\max_{\sigma=\pm1}\sigma a$ and split into $r_2\geq0$ and $r_2\leq0$. On each semicircle the objective is a quadratic plus a linear form. Its stationary equation in $\tan(\phi/2)$, where $\vect{r}=(\cos\phi,\sin\phi)^\T$, has degree at most four. Comparing its real feasible roots with the semicircle endpoints gives a global maximizer. An identically zero stationary equation means that the corresponding arc is constant, and any point on that arc suffices. This also covers repeated-root and constant-objective cases. 
For (ii), maximality gives
$$\tilde q(c^*,s^*)-\tilde q(1,0)\geq\tau(c^*|s^*|+(s^*)^2).$$
Because $c^*,|s^*|\geq0$ and $(c^*)^2+(s^*)^2=1$, we have $c^*+|s^*|\geq1$. Thus the gain is at least $\tau|s^*|\geq\tau\sqrt{1-c^*}$. Finally,
\(\dist(\vect{0},\partial_{\Sphere{1}}\tilde q(1,0))
\leq4|\Gamma_{12}|\leq4M_{\mathrm{R4}}\),
which proves \eqref{eq:real4-penalized-ascent} with exactly the stated constant.
\end{proof}

\begin{proof}[Proof of \cref{prop:uniform-ascent-real4}]
If the candidate is the identity, the assertion is immediate.
Otherwise, the two nonzero selected contractions make the active
branch in \eqref{eq:subpro_hk_real} unique and smooth at \((1,0)\). For the quadratic branch,
\cref{lem:real4-absolute-rayleigh-ascent,lem:real4-penalized-ascent}
give the required gain, according as \eqref{eq:condi_4real} holds or
fails. For the linear branch, write
\(\vect{r}^*=(x^*,y^*)^\T
=\operatorname{sign}(b_1)\vect{b}/\|\vect{b}\|\).
Its gain is \(\|\vect{b}\|(1-x^*)\), its spherical gradient norm at
the identity is \(2\|\vect{b}\||y^*|\), and
\(1-x^*\geq |y^*|\sqrt{1-c^*}\). Thus every branch has a gain bounded
below by the corresponding constant in the statement times its
gradient norm and \(\sqrt{1-c^*}\). Since the selected branch is
uniquely active, the gradient norm of \(h_{k,\mathrm{R4}}\) is twice
that of the branch, while its gain is at least twice the branch gain.
Finally,
\(\|\matr{\Psi}(c^*,s^*)-\matr{I}_2\|_F=2\sqrt{1-c^*}\).
Taking the minimum of the branch constants gives
\(\varepsilon_{\mathrm{R4}}\).
\end{proof}

\subsection{Fourth-order complex case}
For the complex fourth-order model \eqref{cost-fn-complex-l1-diag},
we first give an explicit representation of the two selected diagonal
contractions.

\begin{lemma}
\label{lem:complex4-representation}
Let $\tens{W} \in \CC^{2\times 2\times 2\times 2}$ satisfy
$\tenselem{W}_{ijkl} = \tenselem{W}_{klij}^{*}$ and
$\tenselem{W}_{ijkl} = \tenselem{W}_{jikl}$ for all $1\leq i,j,k,l \leq 2.$ 
Let $\matr{\Psi}(c,s_1,s_2)\in\UN{2}$ be as in \eqref{eq:2.10}, and define
$\tens{T} = \tens{W} \contr{1} \matr{\Psi}^\H \contr{2} \matr{\Psi}^\H \contr{3} \matr{\Psi}^\T \contr{4} \matr{\Psi}^\T.$ 
Define
\begin{equation*}
\vect{r} = \vect{r}(c,s_1,s_2) \eqdef
[2c^2 - 1, 2c s_1, -2c s_2]^{\T} \in \Sphere{2}.
\end{equation*}
Then
$\tenselem{T}_{1111} = \vect{r}^\T \matr{\Gamma} \vect{r} + \vect{b}^\T \vect{r}, 
\tenselem{T}_{2222} = \vect{r}^\T \matr{\Gamma} \vect{r} - \vect{b}^\T \vect{r},$ 
where
\begingroup\small
\begin{align}
\matr{\Gamma} &=
\begin{pmatrix}
\frac{\tenselem{W}_{1111} + \tenselem{W}_{2222}}{2}
&
\frac{\Re(\tenselem{W}_{1112}) - \Re(\tenselem{W}_{1222})}{2}
&
\frac{\Im(\tenselem{W}_{1222})-\Im(\tenselem{W}_{1112})}{2}
\\[2mm]
\frac{\Re(\tenselem{W}_{1112}) - \Re(\tenselem{W}_{1222})}{2}
&
\frac{\tenselem{W}_{1111} + \tenselem{W}_{2222}+2\Re(\tenselem{W}_{1122})}{4}+ \Re(\tenselem{W}_{1212}) 
&
-\frac{\Im(\tenselem{W}_{1122})}{2}
\\[2mm]
\frac{\Im(\tenselem{W}_{1222})-\Im(\tenselem{W}_{1112})}{2}
&
-\frac{\Im(\tenselem{W}_{1122})}{2}
&
\frac{\tenselem{W}_{1111} + \tenselem{W}_{2222}-2\Re(\tenselem{W}_{1122})}{4} + \Re(\tenselem{W}_{1212})
\end{pmatrix},\notag\\
\vect{b} &=
\begin{pmatrix}
\frac{\tenselem{W}_{1111} - \tenselem{W}_{2222}}{2} \\
\Re(\tenselem{W}_{1112}) + \Re(\tenselem{W}_{1222}) \\
-\,\Im(\tenselem{W}_{1112})-\,\Im(\tenselem{W}_{1222})
\end{pmatrix}. \label{eq:4thcomplex_Gamma}
\end{align}
\endgroup
\end{lemma}

For $d=4$ and $r=2$ in \eqref{cost-fn-complex-l1-diag}, the two diagonal
contractions are real. Hence the reduced elementary objective is
\begin{equation}\label{eq:subpro_hk_complex}
h_{k,\mathrm{C4}}(c,s_1,s_2)=h_{k,\mathrm{C4}}(\matr{\Psi})
=2\max\{|\vect{r}^\T\matr{\Gamma}\vect{r}|,\,|\vect{b}^\T\vect{r}|\},
\end{equation}
where $\matr{\Gamma}$ and $\vect{b}$ are constructed from $\tens{W}(\matr{U}_{k-1})^{(i_k,j_k)}$ by \eqref{eq:4thcomplex_Gamma}. The unitarily transformed tensor retains the two symmetries in \cref{lem:complex4-representation}. 
Let $\lambda_1\geq\lambda_2\geq\lambda_3$ be the eigenvalues of $\matr{\Gamma}$, with an orthonormal eigenbasis $\vect{w},\vect{u},\vect{v}$. Whenever an extremal eigenvector is selected below, that same vector is included in the eigenbasis used for the spectral tests. Let $\lambda_*=\lambda_1$ if $|\lambda_1|\geq|\lambda_3|$, and $\lambda_*=\lambda_3$ otherwise. Define
\begin{equation}\label{eq:q_func_cs_comp}
\tilde q(c,s_1,s_2)=|q(c,s_1,s_2)|,\ \ 
\tilde p(c,s_1,s_2)=|\vect{b}^\T\vect{r}(c,s_1,s_2)|,
\end{equation}
where
\begin{equation}\label{eq:q_func_cs_tilde_4com}
q(c,s_1,s_2)=\vect{r}(c,s_1,s_2)^\T\matr{\Gamma}\vect{r}(c,s_1,s_2).
\end{equation}
Fix $0<\tau<1$. The auxiliary penalized problem is defined on the closed hemisphere
$\Sphere{2}_+=\{(c,s_1,s_2)\in\Sphere{2}:c\geq0\}$ by
\begin{equation}\label{eq:tildeq_eta_def_complex}
\tilde q_\tau(c,s_1,s_2)=\tilde q(c,s_1,s_2)
-\tau\bigl(c(|s_1|+|s_2|)+s_1^2+s_2^2\bigr).
\end{equation}
All maximizers of this auxiliary function used by the algorithm below are
over \(\Sphere{2}_+\). The domain restriction is essential: with this
formula the penalty need not be nonnegative on the other hemisphere.
Continuity and compactness give a uniform bound
\(\|\matr{\Gamma}\|_2\leq M_{\mathrm{C4}}\) for some
\(M_{\mathrm{C4}}>0\).

{
\floatname{algorithm}{Subalgorithm}
\renewcommand{\thealgorithm}{CG-C4}
\begin{algorithm}
\caption{Candidate generation for the reduced objective \eqref{eq:subpro_hk_complex}}\label{alg:branch-solver-h-complex}
\begin{algorithmic}[1]
\STATE \textbf{Input:} the elementary function \eqref{eq:subpro_hk_complex}, fixed $\varrho,\upsilon>0$, and $0<\tau<1$.
\STATE \textbf{Output:} a rotation $\matr{\Psi}(c^*,s_1^*,s_2^*)$.
\STATE Form an ordered orthonormal eigenbasis containing an eigenvector for $\lambda_*$ with maximal nonnegative first component; use this same basis in the tests below.
\IF{$|\Gamma_{11}|\geq|b_1|$}
\STATE Choose the quadratic branch.
\IF{$\lambda_1=\lambda_3$}
\STATE Set $(c^*,s_1^*,s_2^*)=(1,0,0)$.
\ELSE
\STATE \textbf{if} both conditions
\begin{equation}\label{eq:cond_4th_complex_eigen_22}
\begin{cases}
\dfrac{\lambda_1-\lambda_2}{\lambda_1-\lambda_3}u_1^2+v_1^2\geq\varrho u_1^2,
&|\lambda_1|\geq|\lambda_3|,\\[2mm]
\dfrac{\lambda_2-\lambda_3}{\lambda_1-\lambda_3}u_1^2+w_1^2\geq\varrho u_1^2,
&|\lambda_3|>|\lambda_1|
\end{cases}
\end{equation}
and
\begin{equation}\label{eq:cond_4th_complex_eigen_232}
|\lambda_1+\lambda_3|\geq\upsilon\sqrt{\Gamma_{12}^2+\Gamma_{13}^2}
\end{equation}
hold, \textbf{then} choose a unit eigenvector $\vect{r}^*$ corresponding to $\lambda_*$ with $r_1^*\geq0$, maximizing $r_1^*$ among tied choices.
\STATE \textbf{else} use an exact oracle to choose a global maximizer
$(c^*,s_1^*,s_2^*)$ of $\tilde q_\tau$ on $\Sphere{2}_+$; if several maximizers exist, any one may be selected.
\ENDIF
\ELSE
\STATE Choose the linear branch and set $\vect{r}^*=\operatorname{sign}(b_1)\vect{b}/\|\vect{b}\|$.
\ENDIF
\STATE Whenever $\vect{r}^*$ was computed, recover $c^*=\sqrt{(1+r_1^*)/2}$, $s_1^*=r_2^*/(2c^*)$, and $s_2^*=-r_3^*/(2c^*)$.
\vspace{-\baselineskip}
\STATE If $h_{k,\mathrm{C4}}(\matr{\Psi}(c^*,s_1^*,s_2^*))=
h_{k,\mathrm{C4}}(\matr{I}_2)$, set $(c^*,s_1^*,s_2^*)=(1,0,0)$.
\STATE Return $\matr{\Psi}(c^*,s_1^*,s_2^*)$.
\end{algorithmic}
\end{algorithm}
}

The selected branch is active at the identity, and the penalty is
nonnegative on $\Sphere{2}_+$ and vanishes at the identity.
Hence every candidate is nondecreasing for the reduced objective.
Any global auxiliary maximizer is admissible.
The fallback uses an exact oracle; its existence follows from
compactness, but a finite implementation is not developed here.

\begin{proposition}
\label{prop:uniform-ascent-complex4}
Let $h_{k,\mathrm{C4}}(c,s_1,s_2)$ be defined as in
\eqref{eq:subpro_hk_complex}. Suppose
$g_{i_k}(\matr{U}_{k-1})g_{j_k}(\matr{U}_{k-1})\neq0$.
Then the candidate returned by
\subalgref{alg:branch-solver-h-complex} satisfies
\[
h_{k,\mathrm{C4}}(c^*,s_1^*,s_2^*)
-h_{k,\mathrm{C4}}(1,0,0)
\geq
\varepsilon_{\mathrm{C4}}\,
\dist\left(\vect{0},
\partial_{\Sphere{2}}h_{k,\mathrm{C4}}(1,0,0)\right)
\left\|\matr{\Psi}(c^*,s_1^*,s_2^*)-\matr{I}_2\right\|_F,
\]
where
\(\varepsilon_{\mathrm{C4}}
\eqdef\frac12\min\{\frac12,
\frac{\varrho}{2\sqrt2(1+\varrho)},\frac{\upsilon}{4},
\frac{\tau}{4M_{\mathrm{C4}}}\}>0\) is independent of \(k\).
\end{proposition}

\begin{lemma}
\label{lem:complex4-rayleigh-ascent}
Let $q(c,s_1,s_2)$ be defined as in \eqref{eq:q_func_cs_tilde_4com}. 
Let \(\vect{w}\eqdef\vect{r}(c^\ast,s_1^\ast,s_2^\ast)\) be a unit eigenvector corresponding to the
largest eigenvalue of \(\matr{\Gamma}\). 
We choose the sign of this eigenvector
such that its first component is nonnegative, and then choose
\((c^\ast,s_1^\ast,s_2^\ast)\in \Sphere{2}\) satisfying \(c^\ast\ge 0\). 
Suppose that there exists $\varrho>0$ such that
\begin{equation}\label{eq:cond_4th_complex_eigen}
\frac{\lambda_1-\lambda_2}{\lambda_1-\lambda_3}u_1^2
+
v_1^2
\geq
\varrho u_1^2
\end{equation}
whenever $\lambda_1\neq \lambda_3$. 
Then there exists a constant \(\varepsilon>0\), independent of
\(\matr{\Gamma}\), such that
\[
q(c^\ast,s_1^\ast,s_2^\ast)-q(1,0,0)
\geq
\varepsilon 
\|\grad q(1,0,0)\|
\sqrt{1-c^*}. 
\]
\end{lemma}

\begin{proof}
If $\lambda_1=\lambda_3$, then $\matr{\Gamma}=\lambda_1\matr{I}_3$, and both the gain and the gradient are zero. Otherwise put $D=\lambda_1-\lambda_3>0$ and $\alpha=(\lambda_1-\lambda_2)/D\in[0,1]$. Since $\vect{e}_1=w_1\vect{w}+u_1\vect{u}+v_1\vect{v}$, the gain is $D(\alpha u_1^2+v_1^2)$. Assumption \eqref{eq:cond_4th_complex_eigen} and the inequality $\alpha u_1^2+v_1^2\geq v_1^2$ give
$\alpha u_1^2+v_1^2\geq\frac{\varrho}{1+\varrho}(u_1^2+v_1^2)$. 
The differential of the coordinate map at $(1,0,0)$ sends $(0,\xi_1,\xi_2)^\T$ to $(0,2\xi_1,-2\xi_2)^\T$. Consequently,
$$\|\grad q(1,0,0)\|=4\sqrt{\Gamma_{12}^2+\Gamma_{13}^2}
=4\|(\matr{I}_3-\vect{e}_1\vect{e}_1^\T)\matr{\Gamma}\vect{e}_1\|.$$
Subtracting $\lambda_1\matr{I}_3$ inside the projection bounds this by
$4D\sqrt{\alpha^2u_1^2+v_1^2}\leq4D\sqrt{u_1^2+v_1^2}$.
Combining these inequalities bounds the gain below by
$\frac{\varrho}{4(1+\varrho)}\|\grad q(1,0,0)\|\sqrt{u_1^2+v_1^2}$. 
Finally, \(w_1=2(c^*)^2-1\geq0\) implies \(c^*\geq1/\sqrt2\).
Moreover,
\(\sqrt{u_1^2+v_1^2}=2c^*\sqrt{1+c^*}\sqrt{1-c^*}\), which is at
least \(\sqrt2\sqrt{1-c^*}\). The stated estimate follows with
\(\varepsilon=\varrho/[2\sqrt2(1+\varrho)]\), independently of
\(\matr{\Gamma}\).
\end{proof}

\begin{remark}
The spectral condition \eqref{eq:cond_4th_complex_eigen} is related to the conditions in \cite[Lemma~7.3]{ULC2019} and \cite[Theorem~5.7]{sheng2023jacobi}. In particular, it holds when $u_1=0$, and also when $|v_1|\geq\sqrt{\varrho}\,|u_1|$.
\end{remark}

\begin{lemma}
\label{lem:complex4-absolute-rayleigh-ascent}
Let $\tilde{q}(c,s_1,s_2)$ be defined as in \eqref{eq:q_func_cs_comp}. 
Let \(\vect{r}^{\ast}\eqdef\vect{r}(c^\ast,s_1^\ast,s_2^\ast)\) be a unit eigenvector corresponding to $\lambda_{\ast}$. 
We choose the sign of this eigenvector
such that its first component is nonnegative, and then choose
\((c^\ast,s_1^\ast,s_2^\ast)\in \Sphere{2}\) satisfying \(c^\ast\ge 0\). 
Suppose that there exist $\varrho>0, \upsilon>0$ such that the conditions \eqref{eq:cond_4th_complex_eigen_22} and \eqref{eq:cond_4th_complex_eigen_232} hold whenever $\lambda_1\neq\lambda_3$.
Then there exists a constant \(\varepsilon>0\), independent of
\(\matr{\Gamma}\), such that
\begin{equation}\label{eq:complex4-absolute-rayleigh-ascent}
\tilde{q}(c^*,s_1^*,s_2^*) - \tilde{q}(1,0,0) 
\;\ge\; 
\varepsilon \; \dist\big(\vect{0}, \partial_{\Sphere{2}} \tilde{q}(1,0,0)\big) \;\sqrt{1-c^*}.
\end{equation}
\end{lemma}

\begin{proof}
If all three eigenvalues are equal, the assertion is immediate. Otherwise set $d_0=\dist(\vect{0},\partial_{\Sphere{2}}\tilde q(1,0,0))$. As in the real case,
$d_0\leq4\sqrt{\Gamma_{12}^2+\Gamma_{13}^2}$, with equality when $\Gamma_{11}\ne0$ and with $d_0=0$ when $\Gamma_{11}=0$. 
First suppose $\Gamma_{11}\lambda_*\geq0$. If $\lambda_*=\lambda_1$, apply \cref{lem:complex4-rayleigh-ascent} to $\matr{\Gamma}$ and the first inequality in \eqref{eq:cond_4th_complex_eigen_22}. If $\lambda_*=\lambda_3$, apply the same lemma to $-\matr{\Gamma}$, whose decreasing eigenvalues and corresponding eigenvectors are $-\lambda_3,-\lambda_2,-\lambda_1$ and $\vect{v},\vect{u},\vect{w}$. Its spectral ratio is precisely the second one in \eqref{eq:cond_4th_complex_eigen_22}. In either case the absolute-value gain is at least
$\frac{\varrho}{2\sqrt2(1+\varrho)}d_0\sqrt{1-c^*}$. 
If $\Gamma_{11}\lambda_*<0$, the Rayleigh bounds and the choice of the eigenvalue of largest modulus give
$|\lambda_*|-|\Gamma_{11}|\geq|\lambda_1+\lambda_3|$.
By \eqref{eq:cond_4th_complex_eigen_232} this is at least $\upsilon d_0/4$, and hence at least $\frac{\upsilon}{4}d_0\sqrt{1-c^*}$. The two cases prove \eqref{eq:complex4-absolute-rayleigh-ascent} with
$\varepsilon=\min\{\varrho/[2\sqrt2(1+\varrho)],\,\upsilon/4\}>0$.
\end{proof}

\begin{lemma}
\label{lem:complex4-penalized-ascent}
Let \(\tilde{q}_{\tau}\) be defined on \(\Sphere{2}_+\) as in
\eqref{eq:tildeq_eta_def_complex}. It admits a global maximizer
\((c^*,s_1^*,s_2^*)\), and every such maximizer satisfies
\begin{equation}\label{eq:complex4-penalized-ascent}
\tilde q(c^*,s_1^*,s_2^*)-\tilde q(1,0,0)
\ge
\frac{\tau}{4M_{\mathrm{C4}}}
\dist\left(
\vect{0},
\partial_{\Sphere{2}}\tilde q(1,0,0)
\right)
\sqrt{1-c^*},
\end{equation}
where $\tau$ is as in \eqref{eq:tildeq_eta_def_complex}.
\end{lemma}

\begin{proof}
Existence follows from continuity on the compact hemisphere
\(\Sphere{2}_+\). Maximality and the fact that the penalty vanishes at
\((1,0,0)\) give
\[
\tilde q(c^*,s_1^*,s_2^*)-\tilde q(1,0,0)
\geq\tau\bigl(c^*(|s_1^*|+|s_2^*|)+(s_1^*)^2+(s_2^*)^2\bigr).
\]
Writing \(t^*=\sqrt{(s_1^*)^2+(s_2^*)^2}\), the right-hand side is at
least \(\tau t^*(c^*+t^*)\geq\tau t^*\geq
\tau\sqrt{1-c^*}\). Moreover,
\(\dist(\vect{0},\partial_{\Sphere{2}}\tilde q(1,0,0))
\leq4\sqrt{\Gamma_{12}^2+\Gamma_{13}^2}
\leq4M_{\mathrm{C4}}\).
Combining these estimates proves \eqref{eq:complex4-penalized-ascent}.
\end{proof}

\begin{proof}[Proof of \cref{prop:uniform-ascent-complex4}]
If the candidate is the identity, the assertion is immediate.
Otherwise, the two nonzero selected contractions make the active
branch in \eqref{eq:subpro_hk_complex} unique and smooth at \((1,0,0)\). For the quadratic
branch,
\cref{lem:complex4-absolute-rayleigh-ascent,lem:complex4-penalized-ascent}
give the required gain, according as the spectral tests hold or fail.
For the linear branch, write
\(\vect{r}^*=(x^*,y^*,z^*)^\T
=\operatorname{sign}(b_1)\vect{b}/\|\vect{b}\|\) and
\(t^*=\sqrt{(y^*)^2+(z^*)^2}\). Its gain is
\(\|\vect{b}\|(1-x^*)\), its spherical gradient norm at the identity
is \(2\|\vect{b}\|t^*\), and
\(1-x^*\geq t^*\sqrt{1-c^*}\). Thus every branch has a gain bounded
below by the corresponding constant in the statement times its
gradient norm and \(\sqrt{1-c^*}\). Since the selected branch is
uniquely active, the gradient norm of \(h_{k,\mathrm{C4}}\) is twice
that of the branch, while its gain is at least twice the branch gain.
Finally,
\(\|\matr{\Psi}(c^*,s_1^*,s_2^*)-\matr{I}_2\|_F
=2\sqrt{1-c^*}\). Taking the minimum of the branch constants gives
\(\varepsilon_{\mathrm{C4}}\).
\end{proof}
\section{Solving the subproblems for the third-order cumulant tensor case}\label{sec:3rd-order-sub-complex-real}

\subsection{Third-order real case}
For the real third-order model \eqref{cost-fn-complex-l1-diag-real},
we first express the two selected diagonal contractions as cubic forms.

\begin{lemma}
\label{lem:real3-representation} 
Let $\tens{W} \in \RR^{2\times 2\times 2}$ be a symmetric tensor. 
Let $\matr{\Psi}(c,s)\in\ON{2}$ be as in \eqref{eq:Given_rotations}, and define
$\tens{T} = \tens{W} \contr{1} \matr{\Psi}^\T \contr{2} \matr{\Psi}^\T \contr{3} \matr{\Psi}^\T.$
Define
$\vect{x} = 
(c, s)^{\T} \in \RR^2.$ 
Then
\begin{align*}
\tenselem{T}_{111} 
&=
\tenselem{W}_{111}c^{3}
+3\tenselem{W}_{112}c^{2}s
+3\tenselem{W}_{122}cs^{2}
+\tenselem{W}_{222}s^{3},
\\
\tenselem{T}_{222} &=
-\tenselem{W}_{111}s^{3}
+3\tenselem{W}_{112}s^{2}c
-3\tenselem{W}_{122}sc^{2}
+\tenselem{W}_{222}c^{3}.
\end{align*}
\end{lemma}

By \cref{lem:real3-representation}, the reduced elementary objective is
\begin{equation}
\label{eq:hk_tensor_form}
h_{k,\mathrm{R3}}(c,s) = h_{k,\mathrm{R3}}(\matr{\Psi})
= \max \Bigl\{
\bigl| \tens{C}^{+} \contr{1} \vect{x}^\T
\contr{2} \vect{x}^\T\contr{3} \vect{x}^\T \bigr|, \;
\bigl| \tens{C}^{-} \contr{1} \vect{x}^\T
\contr{2} \vect{x}^\T\contr{3} \vect{x}^\T \bigr| \Bigr\}.
\end{equation}
The identity $|a|+|b|=\max\{|a+b|,|a-b|\}$ for real $a,b$ gives
\eqref{eq:hk_tensor_form},
where $\tens{C}^{+},\tens{C}^{-}\in\RR^{2\times2\times2}$ are symmetric
tensors defined by the following four independent entries each:
\begin{align*}
\tenselem{C}^{+}_{111} &= \tenselem{W}_{111} + \tenselem{W}_{222},&
\tenselem{C}^{+}_{112} &= \tenselem{W}_{112} - \tenselem{W}_{122},&
\tenselem{C}^{+}_{122} &= \tenselem{W}_{112} + \tenselem{W}_{122},&
\tenselem{C}^{+}_{222} &= -\tenselem{W}_{111} + \tenselem{W}_{222},\\
\tenselem{C}^{-}_{111} &= \tenselem{W}_{111} - \tenselem{W}_{222},&
\tenselem{C}^{-}_{112} &= \tenselem{W}_{112} + \tenselem{W}_{122},&
\tenselem{C}^{-}_{122} &= \tenselem{W}_{122} - \tenselem{W}_{112},&
\tenselem{C}^{-}_{222} &= \tenselem{W}_{111} + \tenselem{W}_{222}.
\end{align*} 
All remaining entries of $\tens{C}^{+}$ and $\tens{C}^{-}$ are obtained by
permuting their indices.
Denote the two branch functions in \eqref{eq:hk_tensor_form} by
$\tilde q(c,s)\eqdef |q(c,s)|$ and
$\tilde p(c,s)\eqdef |p(c,s)|$,
where
\begin{equation}\label{eq:q_func_cs_3rd_real}
q(c,s)\eqdef \tens{C}^{+} \contr{1} \vect{x}^\T
\contr{2} \vect{x}^\T\contr{3} \vect{x}^\T,\ \ 
p(c,s)\eqdef \tens{C}^{-} \contr{1} \vect{x}^\T
\contr{2} \vect{x}^\T\contr{3} \vect{x}^\T. 
\end{equation} 
Fix $0<\tau<1$. We also define two auxiliary functions:
\begin{equation}\label{eq:tildeq_eta_def_3rdorder}
\tilde{q}_{\tau}(c,s)\eqdef 
\tilde{q}(c,s)-\tau|s|,\ \ \tilde{p}_{\tau}(c,s)\eqdef 
\tilde{p}(c,s)-\tau|s|.
\end{equation}

In the present two-dimensional setting, the one-branch auxiliary
problems associated with $q(c,s)$ and $p(c,s)$ in
\eqref{eq:q_func_cs_3rd_real} can be solved by finite algebraic
candidate enumeration, as shown in the proof of
\cref{lem:real3-penalized-ascent}. For the convergence analysis of the
full reduced objective \(h_{k,\mathrm{R3}}\), we use the
branch-selection strategy described in
\subalgref{alg:branch-solver-h-complex-3real}.
There exists a constant \(M_{\mathrm{R3}}>0\) such that
\(\|\tens{C}^{+}\|_F\leq M_{\mathrm{R3}}\) and
\(\|\tens{C}^{-}\|_F\leq M_{\mathrm{R3}}\) throughout the algorithm.
In particular, every tensor entry has magnitude at most \(M_{\mathrm{R3}}\).
These bounds are uniform because
the transformed tensor depends continuously on $\matr{U}$ in the
compact orthogonal group.

{
\floatname{algorithm}{Subalgorithm}
\renewcommand{\thealgorithm}{CG-R3}

\begin{algorithm}
\caption{Candidate generation for the reduced objective \eqref{eq:hk_tensor_form}}\label{alg:branch-solver-h-complex-3real}
\begin{algorithmic}[1]
\STATE \textbf{Input:} the elementary function \eqref{eq:hk_tensor_form} and a fixed constant $0<\tau<1$.
\STATE \textbf{Output:} a rotation matrix
$\matr{\Psi}(c^{\ast}, s^{\ast})$.
\IF{$|\tenselem{C}^{+}_{111}|\geq |\tenselem{C}^{-}_{111}|$}
\STATE Compute a global maximizer $(c^*,s^*)$ of \(\tilde{q}_{\tau}(c,s)\) in \eqref{eq:tildeq_eta_def_3rdorder} over $c^2+s^2=1$, $c\geq0$.
\ELSIF{$|\tenselem{C}^{-}_{111}|>|\tenselem{C}^{+}_{111}|$}
\STATE Compute a global maximizer $(c^*,s^*)$ of \(\tilde{p}_{\tau}(c,s)\) in \eqref{eq:tildeq_eta_def_3rdorder} over $c^2+s^2=1$, $c\geq0$.
\ENDIF
\STATE Return $\matr{\Psi}(c^*,s^*)$.
\end{algorithmic}
\end{algorithm}
} 

The selected branch is active at the identity, while the penalty in
\eqref{eq:tildeq_eta_def_3rdorder} is nonnegative and vanishes there.
Therefore, every output of
\subalgref{alg:branch-solver-h-complex-3real} is nondecreasing for the
true elementary objective.

The next estimate supplies the sufficient ascent required in
\cref{subsec:conver_gence_general}.

\begin{proposition}
\label{prop:uniform-ascent-real3}
Let $h_{k,\mathrm{R3}}(c,s)$ be defined as in
\eqref{eq:hk_tensor_form}. Suppose
$g_{i_k}(\matr{U}_{k-1})g_{j_k}(\matr{U}_{k-1})\neq0$.
Then the candidate returned by
\subalgref{alg:branch-solver-h-complex-3real} satisfies
\[
h_{k,\mathrm{R3}}(c^*,s^*)-h_{k,\mathrm{R3}}(1,0)
\geq
\varepsilon_{\mathrm{R3}}\,
\dist\left(\vect{0},
\partial_{\Sphere{1}}h_{k,\mathrm{R3}}(1,0)\right)
\left\|\matr{\Psi}(c^*,s^*)-\matr{I}_2\right\|_F,
\]
where
\(\varepsilon_{\mathrm{R3}}\eqdef
\tau/(8M_{\mathrm{R3}})>0\) is independent of \(k\).
\end{proposition}

We first establish the corresponding estimate for either penalized branch.

\begin{lemma}
\label{lem:real3-penalized-ascent}
Let $\tilde{q}_{\tau}(c,s)$ be defined as in \eqref{eq:tildeq_eta_def_3rdorder}. 
Then \\
(i) a global maximizer $(c^*,s^*)$ of
\(\tilde{q}_{\tau}\), with $c^*\geq0$, can be obtained from a finite
algebraic candidate set.\\
(ii) we have that 
\[
\tilde q(c^*,s^*)-\tilde q(1,0)
\geq
\frac{\tau}{4M_{\mathrm{R3}}}
\dist
\left(
\vect{0},\partial_{\Sphere{1}}\tilde q(1,0)
\right)
\sqrt{1-c^*},
\]
where $\tau$ is as in \eqref{eq:tildeq_eta_def_3rdorder}.
\end{lemma}

\begin{proof}
For (i), compactness gives a global maximizer, and the identities
$q(-c,-s)=-q(c,s)$ and $|-s|=|s|$ allow one with $c^*\geq0$.
For a finite algebraic construction on the open semicircle $c>0$,
put $t=s/c$ and
$P(t)=\tenselem{C}^{+}_{111}
+3\tenselem{C}^{+}_{112}t
+3\tenselem{C}^{+}_{122}t^2
+\tenselem{C}^{+}_{222}t^3$.
Then $q(c,s)=P(t)/(1+t^2)^{3/2}$. On a region with
$\operatorname{sign}q=\sigma$ and $\operatorname{sign}s=\eta$,
where $\sigma,\eta\in\{-1,1\}$, differentiation with respect to
$\theta$ in $(c,s)=(\cos\theta,\sin\theta)$ gives the stationarity equation
\begin{equation}\label{eq:real3_aux_stationary}
\sigma D(t)-\tau\eta(1+t^2)=0,\ \ 
D(t)=(1+t^2)P'(t)-3tP(t).
\end{equation}
The fourth-degree terms cancel, leaving
$$D(t)=3\tenselem{C}^{+}_{112}
+(6\tenselem{C}^{+}_{122}-3\tenselem{C}^{+}_{111})t
+(3\tenselem{C}^{+}_{222}-6\tenselem{C}^{+}_{112})t^2
-3\tenselem{C}^{+}_{122}t^3.$$
Thus the admissible real roots of \eqref{eq:real3_aux_stationary},
together with the boundary points $s=0$, $q=0$, and $c=0$,
form a finite candidate list. The equation $q=0$ also reduces to
a cubic on $c>0$.

For completeness, identically vanishing polynomials do not invalidate
this enumeration. If $P\equiv0$, the point $(1,0)$ maximizes
$-\tau|s|$. If a stationarity polynomial vanishes identically on a
sign region, the objective is constant there, so its value is attained
at a boundary point by continuity. After these cases are treated,
comparison of the finite candidate values proves (i), including
degenerate cubic or lower-degree equations.

For (ii), optimality at $(c^*,s^*)$ relative to $(1,0)$ gives
\begin{equation}\label{eq:real3_aux_increase}
\tilde q(c^*,s^*)-\tilde q(1,0)\geq\tau|s^*|
\geq\tau\sqrt{1-c^*}.
\end{equation}
The last inequality follows from $c^*\geq0$ and
$|s^*|=\sqrt{(1-c^*)(1+c^*)}$.
The derivative of $q(\cos\theta,\sin\theta)$ at $\theta=0$ is
$3\tenselem{C}^{+}_{112}$. If $\tenselem{C}^{+}_{111}\neq0$, the
spherical subdifferential of
$|q|$ is a singleton of norm $3|\tenselem{C}^{+}_{112}|$.
If $\tenselem{C}^{+}_{111}=0$,
$|q|$ has a local minimum at $(1,0)$, so its regular, and hence
limiting, spherical subdifferential contains zero. In both cases,
$\dist(\vect{0},\partial_{\Sphere{1}}\tilde q(1,0))
\leq3|\tenselem{C}^{+}_{112}|
\leq3M_{\mathrm{R3}}\leq4M_{\mathrm{R3}}$.
Combining this bound with \eqref{eq:real3_aux_increase} proves (ii).
\end{proof}

\begin{proof}[Proof of \cref{prop:uniform-ascent-real3}]
The two nonzero selected contractions make the active branch in
\eqref{eq:hk_tensor_form} unique and smooth at \((1,0)\).
Denote this branch by \(\tilde g\). The estimate in \cref{lem:real3-penalized-ascent}
applies whether \(\tilde g\) is the \(\tilde q\)- or
\(\tilde p\)-branch and gives a gain of at least
\(\tau/(4M_{\mathrm{R3}})\) times its gradient norm and
\(\sqrt{1-c^*}\). Since \(h_{k,\mathrm{R3}}\geq\tilde g\), with
equality at the identity, the same estimate holds for the gain of the
reduced objective. Finally,
\(\|\matr{\Psi}(c^*,s^*)-\matr{I}_2\|_F=2\sqrt{1-c^*}\), which gives
\(\varepsilon_{\mathrm{R3}}=\tau/(8M_{\mathrm{R3}})\).
\end{proof}

\subsection{Third-order complex case}
For the complex third-order model \eqref{cost-fn-complex-l1-diag},
we first express the two selected diagonal contractions as cubic forms.

\begin{lemma}
\label{lem:complex3-representation}
Let $\tens{W} \in \CC^{2\times 2\times 2}$ satisfy
$\tenselem{W}_{ijk} = \tenselem{W}_{ikj}$ for all $1\leq i,j,k\leq 2.$ 
Let $\matr{\Psi}(c,s_1,s_2)\in\UN{2}$ be as in \eqref{eq:2.10}, and define
$\tens{T} = \tens{W} \contr{1} \matr{\Psi}^\H \contr{2} \matr{\Psi}^\T \contr{3} \matr{\Psi}^\T.$ 
Define
$\vect{x} = 
(c, s_1, s_2)^{\T} \in \Sphere{2}.$ 
Then
\begin{equation}
\begin{aligned}
|\tenselem{T}_{111}| + |\tenselem{T}_{222}| 
=
\max_{(\alpha_1,\alpha_2)\in \Sphere{1},\;(\beta_1,\beta_2)\in \Sphere{1}}
\left(\alpha_1\tens{C}_{\mathrm R}+\alpha_2\tens{C}_{\mathrm I}
+\beta_1\tens{S}_{\mathrm R}+\beta_2\tens{S}_{\mathrm I}\right)
\contr{1}\vect{x}^\T\contr{2}\vect{x}^\T\contr{3}\vect{x}^\T,
\end{aligned}
\label{eq:real-third-order-tensor-form}
\end{equation}
where the real tensors
$\tens{C}_{\mathrm R},\tens{C}_{\mathrm I},
\tens{S}_{\mathrm R},\tens{S}_{\mathrm I}
\in \RR^{3\times 3\times 3}$
are defined by
\begin{equation}
\tens{C}_{\mathrm R}\eqdef\operatorname{Re}(\tens{C}),
\ \ 
\tens{C}_{\mathrm I}\eqdef-\operatorname{Im}(\tens{C}),
\ \ 
\tens{S}_{\mathrm R}\eqdef\operatorname{Re}(\tens{S}),
\ \ 
\tens{S}_{\mathrm I}\eqdef-\operatorname{Im}(\tens{S}),
\label{eq:complex3-local-tensors}
\end{equation}
and $\tens{C},\tens{S}\in \CC^{3\times 3\times 3}$ are given by
$\tens{C}\eqdef 
\tens{W}\contr{1} \matr{P}_1^\T
\contr{2} \matr{P}_2^\T
\contr{3} \matr{P}_2^\T$, 
$\tens{S}
=
\tens{W}\contr{1} \matr{Q}_1^\T
\contr{2} \matr{Q}_2^\T
\contr{3} \matr{Q}_2^\T,$ 
with
\[
\matr{P}_1=
\begin{bmatrix}
1&0&0\\
0&1&i
\end{bmatrix},
\ \ 
\matr{P}_2=
\begin{bmatrix}
1&0&0\\
0&1&-i
\end{bmatrix},
\ \ 
\matr{Q}_1=
\begin{bmatrix}
0&-1&i\\
1&0&0
\end{bmatrix},
\ \ 
\matr{Q}_2=
\begin{bmatrix}
0&-1&-i\\
1&0&0
\end{bmatrix}.
\]
\end{lemma}

By \cref{lem:complex3-representation}, the reduced elementary objective is
\begin{equation}
h_{k,\mathrm{C3}}(c,s_1,s_2)
= h_{k,\mathrm{C3}}(\matr{\Psi}) =
\max_{(\alpha_1,\alpha_2)\in\Sphere{1},\;(\beta_1,\beta_2)\in\Sphere{1}}
\widehat h_{k,\mathrm{C3}}
\bigl((c,s_1,s_2),(\alpha_1,\alpha_2),(\beta_1,\beta_2)\bigr).
\label{eq:hk-d3-r1}
\end{equation}
The auxiliary function is defined as
\begin{align}
\widehat h_{k,\mathrm{C3}}:
\Sphere{2}\times \Sphere{1}\times \Sphere{1}
&\to \RR, \notag\\
\bigl((c,s_1,s_2),(\alpha_1,\alpha_2),(\beta_1,\beta_2)\bigr)
&\mapsto
\left(
\alpha_1\tens{C}_{\mathrm R}+\alpha_2\tens{C}_{\mathrm I}
+\beta_1\tens{S}_{\mathrm R}+\beta_2\tens{S}_{\mathrm I}
\right)
\contr{1}\vect{x}^{\T}
\contr{2}\vect{x}^{\T}
\contr{3}\vect{x}^{\T}.
\label{eq:hk-d3-r1-aux}
\end{align}
Here, the real coefficient tensors are constructed according to
\eqref{eq:complex3-local-tensors} from
\(\tens{W}(\matr{U}_{k-1})^{(i_k,j_k)}\).
Write \(\vect{x}=(c,s_1,s_2)^\T\),
\(\vect{\alpha}=(\alpha_1,\alpha_2)^\T\), and
\(\vect{\beta}=(\beta_1,\beta_2)^\T\).
For \(\ell=1,2,3\), let \(\widehat h_{k,\mathrm{C3},\ell}\)
denote the auxiliary function viewed as a function of its \(\ell\)th
block, with the other blocks fixed.
We generate a candidate by \emph{block coordinate ascent} using the
gradient-based block selection rule in
\subalgref{alg:branch-solver-h-complex-3rd}.
The first block uses one \emph{shifted power} update on \(\Sphere{2}\)
\cite{kolda2011shifted,li2019polar}, while the two phase blocks on
\(\Sphere{1}\) admit closed-form solutions.
The finite sequence of updates generates a candidate without requiring
a global solution of the first-block subproblem.

The uniform constants can be computed directly from the fixed input
third-order tensor \(\tens{A}\) in
\eqref{cost-fn-complex-l1-diag}. One admissible choice is
\begin{equation}\label{eq:complex3-computable-parameters}
C_{\mathrm{C3}}=\max\{1,8\|\tens{A}\|_F\},\ \ 
M_{\mathrm{C3}}=3C_{\mathrm{C3}},\ \ 
L_{\mathrm{C3}}=6C_{\mathrm{C3}},\ \ 
\gamma_{\mathrm{sh}}=7C_{\mathrm{C3}}.
\end{equation}
The displayed value of \(\gamma_{\mathrm{sh}}\) is a default; any fixed
\(\gamma_{\mathrm{sh}}>6C_{\mathrm{C3}}
=\max\{M_{\mathrm{C3}},L_{\mathrm{C3}}\}\) is admissible. 
Since \(\|\matr{P}_j\|_2=\|\matr{Q}_j\|_2=\sqrt2\) for \(j=1,2\),
unitary invariance of the Frobenius norm and Cauchy--Schwarz give
\[
\|\tens{C}_{\mathrm R}\|_F+\|\tens{C}_{\mathrm I}\|_F+
\|\tens{S}_{\mathrm R}\|_F+\|\tens{S}_{\mathrm I}\|_F
\leq8\|\tens{A}\|_F\leq C_{\mathrm{C3}}.
\]
For fixed phase blocks, write
\(\tens{H}=\alpha_1\tens{C}_{\mathrm R}
+\alpha_2\tens{C}_{\mathrm I}
+\beta_1\tens{S}_{\mathrm R}
+\beta_2\tens{S}_{\mathrm I}\).
Then \(\|\tens{H}\|_F\leq C_{\mathrm{C3}}\).
For a third-order tensor $\tens{B}$, write
$\tens{B}[\vect{u},\vect{v},\vect{w}]
\eqdef \tens{B}\contr{1}\vect{u}^{\T}
\contr{2}\vect{v}^{\T}\contr{3}\vect{w}^{\T}$.
Direct differentiation of
\(\widehat h_{k,\mathrm{C3},1}(\vect{x})
=\tens{H}[\vect{x},\vect{x},\vect{x}]\) gives
\(\|\nabla\widehat h_{k,\mathrm{C3},1}(\vect{x})\|
\leq M_{\mathrm{C3}}\) and
\(\|\nabla^2\widehat h_{k,\mathrm{C3},1}(\vect{x})\|_2
\leq L_{\mathrm{C3}}\) on the unit ball, without requiring
\(\tens{H}\) to be symmetric.
Thus the first-block gradient is \(L_{\mathrm{C3}}\)-Lipschitz there.
Both phase-coefficient vectors have norm at most
\(C_{\mathrm{C3}}\leq M_{\mathrm{C3}}\).
These bounds are uniform in \(k\) and the phase blocks, so
\eqref{eq:complex3-computable-parameters} ensures a nonzero
normalization denominator and ascent in every shifted power step.

Every call starts from the identity rotation
\(\vect{x}^{(0)}=\vect{e}_1\) and phases maximizing the auxiliary
function at that point. Thus the auxiliary and reduced objectives
agree initially, and auxiliary nondecrease guarantees nondecrease
of the reduced objective.

{
\floatname{algorithm}{Subalgorithm}
\renewcommand{\thealgorithm}{CG-C3}

\begin{algorithm}
\caption{Finite-step candidate generation for \eqref{eq:hk-d3-r1}}
\label{alg:branch-solver-h-complex-3rd}
\begin{algorithmic}[1]
\STATE \textbf{Input:} the auxiliary function
\(\widehat h_{k,\mathrm{C3}}\) in
\eqref{eq:hk-d3-r1-aux}, constants
$0<\delta_{\mathrm{blk}}<1/\sqrt3$ and
$\gamma_{\mathrm{sh}}>\max\{M_{\mathrm{C3}},L_{\mathrm{C3}}\}$,
with the computable constants and default shift in
\eqref{eq:complex3-computable-parameters}, and a fixed integer
$N_{\rm in}\geq1$ independent of the outer iteration \(k\).
\STATE \textbf{Output:} A rotation matrix
\(\matr{\Psi}(c^\ast,s_1^\ast,s_2^\ast)\).
\STATE Set $\vect{x}^{(0)}=\vect{e}_1$. Choose
$\vect{\alpha}^{(0)}$ and $\vect{\beta}^{(0)}$ to maximize their
respective linear phase functions at $\vect{e}_1$; for a zero
coefficient vector choose $(1,0)^\T$. Set
$\vect{z}^{(0)}=(\vect{x}^{(0)},\vect{\alpha}^{(0)},\vect{\beta}^{(0)})$.

\FOR{$\ell=1,\ldots,N_{\rm in}$}
\STATE Choose a block index \(p_\ell\in\{1,2,3\}\) such that
\begin{equation}\label{eq:bcd-block-selection-complex3}
\left\|
\grad \widehat h_{k,\mathrm{C3},p_\ell}
\right\|
\geq
\delta_{\mathrm{blk}}
\left\|
\grad \widehat h_{k,\mathrm{C3}}
\bigl(\vect{z}^{(\ell-1)}\bigr)
\right\|,
\end{equation}
where the gradients are evaluated at
\(\vect{z}^{(\ell-1)}
=
(\vect{x}^{(\ell-1)},\vect{\alpha}^{(\ell-1)},\vect{\beta}^{(\ell-1)})\).

\STATE \textbf{if} $p_\ell=1$, \textbf{then} compute one shifted power iteration for the first block:
\begin{equation}\label{eq:shifted-power-update-complex3}
\vect{x}^{(\ell)}
=
\frac{
\nabla \widehat h_{k,\mathrm{C3},1}
\bigl(
\vect{x}^{(\ell-1)}
\bigr)
+
\gamma_{\mathrm{sh}} \vect{x}^{(\ell-1)}
}{
\left\|
\nabla \widehat h_{k,\mathrm{C3},1}
\bigl(
\vect{x}^{(\ell-1)}
\bigr)
+
\gamma_{\mathrm{sh}} \vect{x}^{(\ell-1)}
\right\|
},
\end{equation}
and set
$\vect{\alpha}^{(\ell)}=\vect{\alpha}^{(\ell-1)}$ and $\vect{\beta}^{(\ell)}=\vect{\beta}^{(\ell-1)}.$

\STATE \textbf{if} $p_\ell=2$, \textbf{then} compute a closed-form global maximizer
\begin{equation}\label{eq:inner_update_3complex}
\vect{\alpha}^{(\ell)}
\in
\argmax_{\vect{\alpha}\in\Sphere{1}}
\widehat h_{k,\mathrm{C3}}
\bigl(
\vect{x}^{(\ell-1)},
\vect{\alpha},
\vect{\beta}^{(\ell-1)}
\bigr),
\end{equation}
where the phase is unchanged if its coefficient vector is zero;
and set $\vect{x}^{(\ell)}=\vect{x}^{(\ell-1)}$ and $\vect{\beta}^{(\ell)}=\vect{\beta}^{(\ell-1)}.$
\STATE \textbf{if} $p_\ell=3$, \textbf{then} compute a closed-form global maximizer
\[
\vect{\beta}^{(\ell)}
\in
\argmax_{\vect{\beta}\in\Sphere{1}}
\widehat h_{k,\mathrm{C3}}
\bigl(
\vect{x}^{(\ell-1)},
\vect{\alpha}^{(\ell-1)},
\vect{\beta}
\bigr),
\]
where the phase is unchanged if its coefficient vector is zero;
and set $\vect{x}^{(\ell)}=\vect{x}^{(\ell-1)}$ and $\vect{\alpha}^{(\ell)}=\vect{\alpha}^{(\ell-1)}.$
\ENDFOR
\STATE Set $\widehat{\vect{x}}=\vect{x}^{(N_{\rm in})}$.
If its first coordinate is negative, replace
$\widehat{\vect{x}}$ by $-\widehat{\vect{x}}$; write
$\widehat{\vect{x}}=(c^*,s_1^*,s_2^*)^\T$.
\STATE Return $\matr{\Psi}(c^*,s_1^*,s_2^*)$.
\end{algorithmic}
\end{algorithm}
}

The next estimate supplies the sufficient ascent required in
\cref{subsec:conver_gence_general}.

\begin{proposition}
\label{prop:uniform-ascent-complex3}
Let $h_{k,\mathrm{C3}}(c,s_1,s_2)$ be defined as in
\eqref{eq:hk-d3-r1}. Suppose
$g_{i_k}(\matr{U}_{k-1})g_{j_k}(\matr{U}_{k-1})\neq0$.
Then the candidate returned by
\subalgref{alg:branch-solver-h-complex-3rd} satisfies
\[
h_{k,\mathrm{C3}}(c^*,s_1^*,s_2^*)
-h_{k,\mathrm{C3}}(1,0,0)
\geq
\varepsilon_{\mathrm{C3}}\,
\dist\left(\vect{0},
\partial_{\Sphere{2}}h_{k,\mathrm{C3}}(1,0,0)\right)
\left\|\matr{\Psi}(c^*,s_1^*,s_2^*)-\matr{I}_2\right\|_F,
\]
where
\(\varepsilon_{\mathrm{C3}}\eqdef
(\gamma_{\mathrm{sh}}-L_{\mathrm{C3}})/
[2\sqrt{2N_{\rm in}}(M_{\mathrm{C3}}+\gamma_{\mathrm{sh}})]>0\)
is independent of \(k\).
\end{proposition}

We first estimate the gain and step length of a shifted power update.

\begin{lemma}
\label{lem:complex3-shifted-power-ascent}
In \subalgref{alg:branch-solver-h-complex-3rd}, when $p_{\ell}=1$, the iterate $\vect{x}^{(\ell)}$ defined by \eqref{eq:shifted-power-update-complex3} satisfies
\begin{align*}
\widehat h_{k,\mathrm{C3},1}(\vect{x}^{(\ell)})
-\widehat h_{k,\mathrm{C3},1}(\vect{x}^{(\ell-1)})
&\geq
\frac{\gamma_{\mathrm{sh}}-L_{\mathrm{C3}}}{2}
\|\vect{x}^{(\ell)}-\vect{x}^{(\ell-1)}\|^2,\\
\|\vect{x}^{(\ell)}-\vect{x}^{(\ell-1)}\|
&\geq\frac{1}{M_{\mathrm{C3}}+\gamma_{\mathrm{sh}}}
\|\grad \widehat h_{k,\mathrm{C3},1}
\bigl(
\vect{x}^{(\ell-1)}
\bigr)\|.
\end{align*}
\end{lemma}

\begin{proof}
Keep the phase blocks fixed and abbreviate
$H=\widehat h_{k,\mathrm{C3},1}$,
$\vect{x}=\vect{x}^{(\ell-1)}$,
$\vect{y}=\vect{x}^{(\ell)}$, and
$\vect{w}=\nabla H(\vect{x})+\gamma_{\mathrm{sh}}\vect{x}$.
The choice $\gamma_{\mathrm{sh}}>M_{\mathrm{C3}}$ gives
$\|\vect{w}\|\geq\gamma_{\mathrm{sh}}-M_{\mathrm{C3}}>0$, so the update
$\vect{y}=\vect{w}/\|\vect{w}\|$ is well defined.
It maximizes $\langle\vect{w},\cdot\rangle$ over the sphere.
Comparison with $\vect{x}$ yields
$$\langle\nabla H(\vect{x}),\vect{y}-\vect{x}\rangle
\geq\gamma_{\mathrm{sh}}(1-\langle\vect{x},\vect{y}\rangle)
=\frac{\gamma_{\mathrm{sh}}}{2}\|\vect{y}-\vect{x}\|^2.$$
The segment joining $\vect{x}$ and $\vect{y}$ is in the unit ball;
the Lipschitz-gradient bound there gives
$H(\vect{y})-H(\vect{x})\geq
\langle\nabla H(\vect{x}),\vect{y}-\vect{x}\rangle
-\frac{L_{\mathrm{C3}}}2\|\vect{y}-\vect{x}\|^2$.
Combining these inequalities proves the first estimate. 
For the second, let
$\matr{P}_{\vect{x}}=\matr{I}_3-\vect{x}\vect{x}^\T$.
Projecting
$\|\vect{w}\|\vect{y}
=\nabla H(\vect{x})+\gamma_{\mathrm{sh}}\vect{x}$
onto the tangent space at $\vect{x}$ gives
$\grad H(\vect{x})
=\|\vect{w}\|\matr{P}_{\vect{x}}(\vect{y}-\vect{x})$.
Thus
$\|\grad H(\vect{x})\|
\leq(M_{\mathrm{C3}}+\gamma_{\mathrm{sh}})
\|\vect{y}-\vect{x}\|$,
as required.
\end{proof}

In \subalgref{alg:branch-solver-h-complex-3rd}, when $p_{\ell}=2$, define
$\vect{a}=(\tens{C}_{\mathrm R}[\vect{x},\vect{x},\vect{x}],
\tens{C}_{\mathrm I}[\vect{x},\vect{x},\vect{x}])^\T
\in\RR^2$, where $\vect{x}=\vect{x}^{(\ell-1)}$.

\begin{lemma}
\label{lem:complex3-phase-ascent}
In \subalgref{alg:branch-solver-h-complex-3rd}, when $p_{\ell}=2$, the iterate $\vect{\alpha}^{(\ell)}$ defined by \eqref{eq:inner_update_3complex} satisfies
\begin{align}
\widehat h_{k,\mathrm{C3},2}(\vect{\alpha}^{(\ell)})
-\widehat h_{k,\mathrm{C3},2}(\vect{\alpha}^{(\ell-1)})
&=
\frac{\|\vect{a}\|}{2}
\|\vect{\alpha}^{(\ell)}-\vect{\alpha}^{(\ell-1)}\|^2,
\label{eq:complex3-phase-gain-corrected}\\
\|\vect{\alpha}^{(\ell)}-\vect{\alpha}^{(\ell-1)}\|
&\geq\frac{1}{M_{\mathrm{C3}}+\gamma_{\mathrm{sh}}}
\|\grad \widehat h_{k,\mathrm{C3},2}
\bigl(
\vect{\alpha}^{(\ell-1)}
\bigr)\|.\notag
\end{align}
\end{lemma}

\begin{proof}
If \(\vect{a}=0\), the phase is unchanged and both assertions are
immediate. Otherwise the update is
\(\vect{\alpha}^{(\ell)}=\vect{a}/\|\vect{a}\|\), and linearity of the
phase block gives \eqref{eq:complex3-phase-gain-corrected}. Moreover,
the spherical gradient at \(\vect{\alpha}^{(\ell-1)}\) has norm at most
\(\|\vect{a}\|\|\vect{\alpha}^{(\ell)}-
\vect{\alpha}^{(\ell-1)}\|\). Since
\(\|\vect{a}\|\leq M_{\mathrm{C3}}\),
the second estimate follows. The same argument applies to the
\(\vect{\beta}\)-block.
\end{proof}

\begin{proof}[Proof of \cref{prop:uniform-ascent-complex3}]
At each inner iterate, the squared product-manifold gradient norm is
the sum of the squared block-gradient norms. Hence a block satisfying
\eqref{eq:bcd-block-selection-complex3} exists because
\(\delta_{\mathrm{blk}}<1/\sqrt3\). The initialization matches the
two maximizing phases at \(\vect{e}_1\), so
\(\widehat h_{k,\mathrm{C3}}(\vect{z}^{(0)})
=h_{k,\mathrm{C3}}(1,0,0)\).
\Cref{lem:complex3-shifted-power-ascent,lem:complex3-phase-ascent}
show that every inner step is nondecreasing for the auxiliary
function. Its final value is no larger than
\(h_{k,\mathrm{C3}}\) at the returned first block, and the final sign
normalization does not change this reduced objective.

The two selected contractions are nonzero at \(\vect{e}_1\), so the
initial maximizing phases are unique.
Differentiating \(|z|\) at \(z\ne0\) therefore shows that the
subdifferential distance of \(h_{k,\mathrm{C3}}\) at the identity
equals the norm of the first-block gradient of
\(\widehat h_{k,\mathrm{C3}}\) at \(\vect{z}^{(0)}\); denote their
common value by \(g_k\).
If \(g_k=0\), the phase gradients also vanish at
\(\vect{z}^{(0)}\), while the first-block ambient gradient is parallel
to \(\vect{e}_1\). Since
\(\gamma_{\mathrm{sh}}>M_{\mathrm{C3}}\), a first-block update then
leaves \(\vect{e}_1\) unchanged; the maximizing phase updates do so as
well. Thus, if \(g_k=0\), or if no first-block update is performed, the
returned first block is \(\vect{e}_1\), and the assertion follows from
monotonicity.
Otherwise, both phase-block gradients vanish at
\(\vect{z}^{(0)}\), so the first performed update is a first-block
update. Put
\(a=(\gamma_{\mathrm{sh}}-L_{\mathrm{C3}})/2>0\), and let \(S\)
be the sum of the squared lengths of all first-block steps. By
\cref{lem:complex3-shifted-power-ascent}, the total gain is at least
\(aS\), and the first step has length at least
\(g_k/(M_{\mathrm{C3}}+\gamma_{\mathrm{sh}})\).
If \(D\) is the distance from the returned sphere parameter to
\(\vect{e}_1\), sign normalization can only decrease this distance,
while the triangle and Cauchy--Schwarz inequalities give
\(D\leq\sqrt{N_{\rm in}S}\). Since \(S\) is also at least the square
of the first-step length, the gain is at least
\(a\,g_kD/
[\sqrt{N_{\rm in}}(M_{\mathrm{C3}}+\gamma_{\mathrm{sh}})]\).
Finally,
\(\|\matr{\Psi}(c^*,s_1^*,s_2^*)-\matr{I}_2\|_F=\sqrt2D\),
which yields the stated \(\varepsilon_{\mathrm{C3}}\).
\end{proof}

\section{Numerical experiments}\label{sec:numerical}

\subsection{Practical implementation of Jacobi methods}
\label{subsec:batch-jacobi-subg}
We evaluate the proposed methods on real fourth-order ICA and \(\ell_1\)-based tensor
approximate diagonalization problems. The Jacobi-SubG variants use
\subalgref{alg:branch-solver-h}.
We consider four implementations. The threshold-based variant
Jacobi-SubG(\(\delta\)) scans the
unordered pairs cyclically and accepts the first pair satisfying
\eqref{inequality-gra-based-2}; Jacobi-SubG-Max selects the pair with the
largest pairwise subgradient norm. Jacobi-Cyclic visits every pair once per
sweep and globally solves each unpenalized local subproblem. Jacobi-CyclicB
divides a sweep into batches of \(q=\lfloor n/2\rfloor\) disjoint pairs,
computes their rotations from the same transformed tensor, and applies them
sequentially. Since disjoint rotations commute and do not change one
another's principal pair subtensors, Jacobi-Cyclic and Jacobi-CyclicB agree
at batch boundaries in exact arithmetic when the same schedule and
deterministic local solver are used.
Since the cyclic rules do not enforce \eqref{inequality-gra-based-2}, the
convergence theorem for Jacobi-SubG does not directly apply to them. 

\Cref{fig:jacobi-batch-kernel-timing} compares the two cyclic
implementations on 30 paired real fourth-order ICA problems for each
\(n\in\{4,8,12,16\}\). Batching reduces the mean time per sweep at every
tested dimension through vectorized coefficient construction and shared
postprocessing. These sequential timings do not measure parallel speedup.

\begin{figure}[htbp]
\centering
\includegraphics[width=\linewidth]{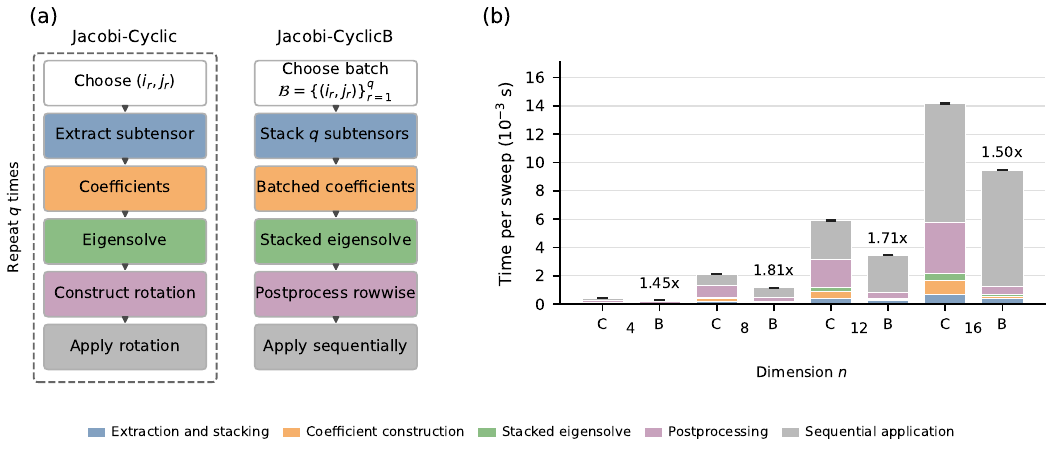}
\caption{(a) One computational stage of Jacobi-Cyclic and Jacobi-CyclicB.
The Jacobi-Cyclic loop selects the next pair and applies its rotation.
(b) Component times per sweep (C: Jacobi-Cyclic; B: Jacobi-CyclicB).
The labels give the mean ratio of the Jacobi-Cyclic time to the
Jacobi-CyclicB time over paired trials.}
\label{fig:jacobi-batch-kernel-timing}
\end{figure}

\subsection{Independent component analysis}
\label{subsec:exp1-ica}

We compare the methods on eight independent unit-variance sources: four
\(\operatorname{Laplace}(0,1/\sqrt{2})\) and four
\(\operatorname{Uniform}[-\sqrt{3},\sqrt{3}]\).
The observations are \(\matr{Y}=\matr{A}\matr{S}+\sigma\matr{G}\), where
\(\matr{A}\in\ON{8}\) is Haar-distributed and
\(G_{ij}\overset{\mathrm{iid}}{\sim}\mathcal{N}(0,1)\).
We center and whiten the observations using their sample covariance.

We vary \(\sigma\in\{0.1,0.2,0.3,0.4,0.5\}\) at \(T=3200\),
and \(T\in\{1600,3200,6400,12800,25600\}\) at \(\sigma=0.5\).
Each setting uses 100 paired trials. Within each trial, all methods share
the observations and independently generated Haar-distributed initial directions.

We compare Jacobi-SubG(\(\delta\)), with
\(\delta\in\{10^{-4},1/\sqrt{n(n-1)}\}\), Jacobi-SubG-Max, and Jacobi-CyclicB,
as in \cref{subsec:batch-jacobi-subg}.
The baselines are Jacobi-CoM2 \cite{comon1994independent,LUC2017globally},
FastICA \cite{hyvarinen1997fast,hyvarinen1999fast}, and
RobustICA \cite{zarzoso2009robust}.
The Jacobi methods use the empirical fourth-order cumulant tensor.
The SubG variants use \subalgref{alg:branch-solver-h} with
\(\tau=10^{-3}\) and \(\upsilon=10^{-8}\).
The Jacobi methods use a relative-update tolerance of \(10^{-7}\),
tested after each sweep or block of \(n(n-1)/2\) pair updates;
the ICA baselines use their componentwise tests at the same tolerance.

We report the Amari error of \(\matr{M}=\matr{B}\matr{A}\), where
\(\matr{B}\) includes the estimated separation and sample whitening:
\[
\operatorname{Amari}(\matr{M})
=\frac{1}{2n(n-1)}
\left[
\sum_i\left(\frac{\sum_j|M_{ij}|}{\max_j|M_{ij}|}-1\right)
+\sum_j\left(\frac{\sum_i|M_{ij}|}{\max_i|M_{ij}|}-1\right)
\right].
\]
All methods satisfy their stopping tests.

\Cref{fig:exp1-ica} shows nearly identical mean errors for the SubG variants
and Jacobi-CyclicB throughout the noise sweep. These four variants have lower
mean Amari error than Jacobi-CoM2 and the ICA baselines at every tested
setting. Increasing the sample size reduces the mean Amari error for all methods.

For these implementations, FastICA has the shortest mean total time.
Jacobi-CyclicB has lower mean total time than the SubG variants throughout both sweeps.

\begin{figure}[htbp]
\centering
\includegraphics[width=\linewidth]{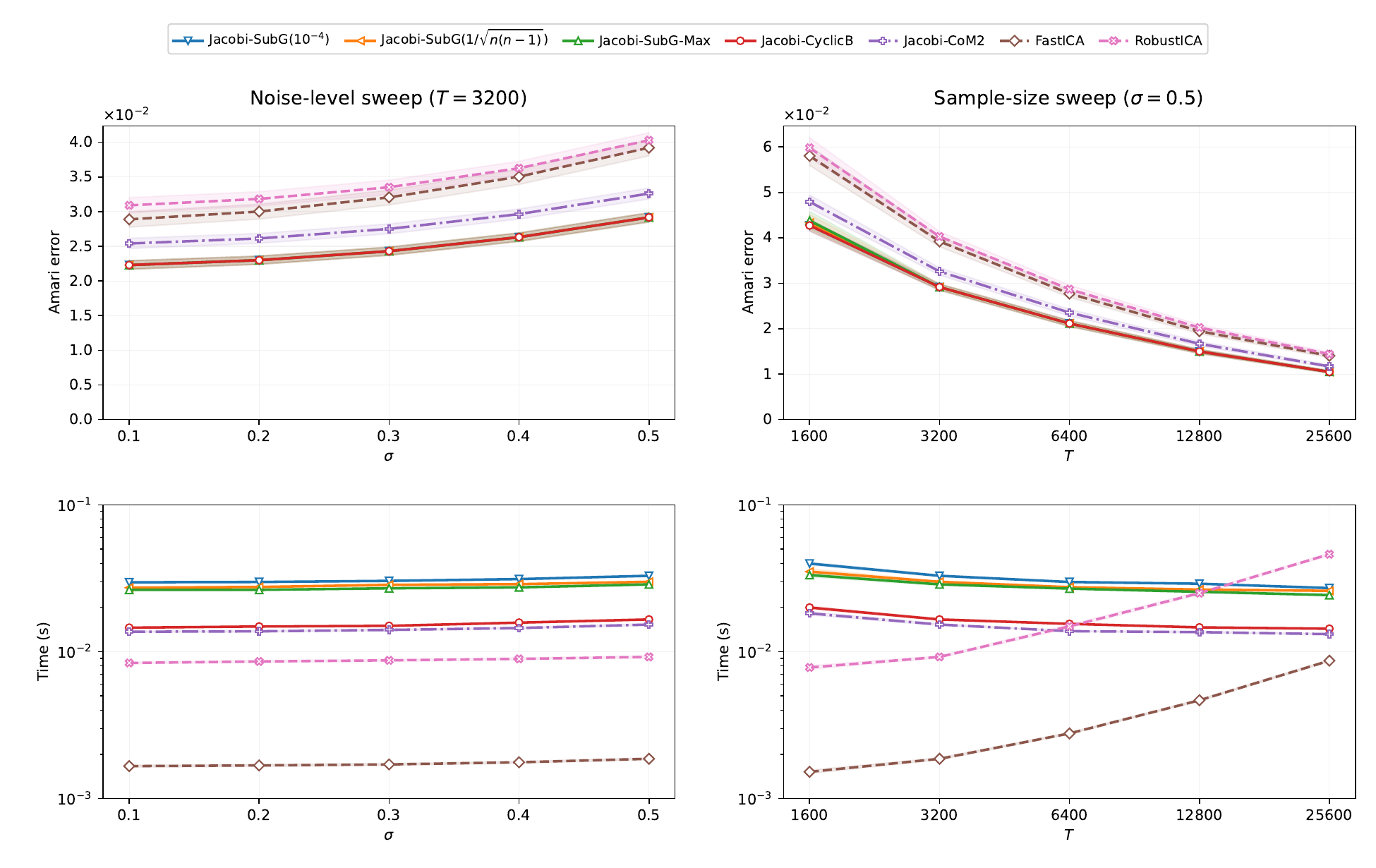}
\caption{Comparison of mean Amari error and mean total computation time in the ICA experiment. Shaded regions show pointwise 95\% bootstrap
confidence intervals over 100 paired trials per setting.}
\label{fig:exp1-ica}
\end{figure}

We report the branch-selection frequencies of the SubG variants and how often
the returned rotation globally maximizes the unpenalized local objective \(h_{k,\mathrm{R4}}\).
Let \(h^\star\) denote its global maximum
over \(c^2+s^2=1\), \(c\geq0\).
Exactness is assessed using a relative tolerance of \(10^{-10}\),
scaled by the larger of \(|h^\star|\) and
\(\|\tens{W}(\matr{U}_{k-1})^{(i_k,j_k)}\|_F\).
As shown in \cref{tab:exp1-branch}, the linear and quadratic-eigenvector
branches of \subalgref{alg:branch-solver-h} are selected in roughly equal proportions.
Nearly all returned rotations globally maximize \(h_{k,\mathrm{R4}}\)
within the stated tolerance. The penalized fallback is never selected.
No angle refinement is triggered in these experiments.

\begin{table}[htbp]
\centering
\caption{Branch-selection and final exact-selection percentages, averaged over 100 trials per setting. Exactness is evaluated after candidate refinement.}
\label{tab:exp1-branch}
\scriptsize
\setlength{\tabcolsep}{2.1pt}
\renewcommand{\arraystretch}{1.15}
\begin{tabular}{@{}lrrrrrrrrrrrr@{}}
\toprule
& \multicolumn{4}{c}{Jacobi-SubG-Max}
& \multicolumn{4}{c}{Jacobi-SubG(\(10^{-4}\))}
& \multicolumn{4}{c}{Jacobi-SubG(\(1/\sqrt{n(n-1)}\))} \\
\cmidrule(lr){2-5}\cmidrule(lr){6-9}\cmidrule(lr){10-13}
Setting & Linear & \shortstack{Quadratic\\eigen} & \shortstack{Penalized\\fallback} & Exact & Linear & \shortstack{Quadratic\\eigen} & \shortstack{Penalized\\fallback} & Exact & Linear & \shortstack{Quadratic\\eigen} & \shortstack{Penalized\\fallback} & Exact \\
\midrule
\multicolumn{13}{c}{Sample-size sweep: \(\sigma=0.5\)} \\
\midrule
\(T=1600\) & 56.1 & 43.9 & 0.0 & 99.1 & 56.5 & 43.5 & 0.0 & 99.0 & 57.0 & 43.0 & 0.0 & 99.1 \\
\(T=3200\) & 56.1 & 43.9 & 0.0 & 99.2 & 56.5 & 43.5 & 0.0 & 98.9 & 56.6 & 43.4 & 0.0 & 99.1 \\
\(T=6400\) & 55.7 & 44.3 & 0.0 & 99.2 & 56.3 & 43.7 & 0.0 & 98.9 & 56.1 & 43.9 & 0.0 & 99.1 \\
\(T=12800\) & 56.3 & 43.7 & 0.0 & 99.1 & 56.2 & 43.8 & 0.0 & 98.7 & 56.5 & 43.5 & 0.0 & 99.0 \\
\(T=25600\) & 55.4 & 44.6 & 0.0 & 99.1 & 56.2 & 43.8 & 0.0 & 98.6 & 56.9 & 43.1 & 0.0 & 99.0 \\
\midrule
\multicolumn{13}{c}{Noise sweep: \(T=3200\)} \\
\midrule
\(\sigma=0.1\) & 57.1 & 42.9 & 0.0 & 99.2 & 56.3 & 43.7 & 0.0 & 98.9 & 56.9 & 43.1 & 0.0 & 99.0 \\
\(\sigma=0.2\) & 56.8 & 43.2 & 0.0 & 99.2 & 56.4 & 43.6 & 0.0 & 98.9 & 56.5 & 43.5 & 0.0 & 99.0 \\
\(\sigma=0.3\) & 55.1 & 44.9 & 0.0 & 99.2 & 56.4 & 43.6 & 0.0 & 98.9 & 56.5 & 43.5 & 0.0 & 99.1 \\
\(\sigma=0.4\) & 56.0 & 44.0 & 0.0 & 99.1 & 56.4 & 43.6 & 0.0 & 98.9 & 56.6 & 43.4 & 0.0 & 99.1 \\
\(\sigma=0.5\) & 56.1 & 43.9 & 0.0 & 99.2 & 56.5 & 43.5 & 0.0 & 98.9 & 56.6 & 43.4 & 0.0 & 99.1 \\
\bottomrule
\end{tabular}
\end{table}

\subsection{\texorpdfstring{\(\ell_1\)}{l1}-based tensor approximate diagonalization}
\label{subsec:exp2-tensor}

We construct symmetric fourth-order cores with sparse negative fourth-order
cross-cumulants.
Each coordinate pair \(i<j\) is included independently in a graph
\(\mathscr E\) with probability \(\rho\). Draw independent
\(c_i\sim\operatorname{Uniform}(0.1,0.3)\) and
\(w_{ij}\sim\operatorname{Uniform}(0.5,1.5)\), with \(w_{ji}=w_{ij}\), and set
\(\tenselem{D}_{iiii}=c_i+\kappa\sum_{j:\{i,j\}\in\mathscr E}w_{ij}\)
and \(\tenselem{D}_{iijj}=-w_{ij}\) on graph edges, including all
symmetric permutations; other entries are zero. Here \(\kappa=2\mathbb E[R^4]-3\), where \(R\) is a randomly signed
\(\operatorname{Uniform}(0.1,1)\) amplitude standardized to unit second moment.

For an independent Haar-distributed \(\matr{Q}\in\ON{n}\), define
\(\tens{A}_0=(\tens{D}\contr{1}\matr{Q}\contr{2}\matr{Q}\contr{3}\matr{Q}\contr{4}\matr{Q})/\|\tens{D}\|_F\).
Let \(\tens{E}=\operatorname{sym}_4(\tens{Z})/\|\operatorname{sym}_4(\tens{Z})\|_F\),
where \(\tens{Z}\) has independent standard Gaussian entries and
\(\operatorname{sym}_4\) averages over all mode permutations. The observed
tensor is \(\tens{A}=(\tens{A}_0+\sigma\tens{E})/\|\tens{A}_0+\sigma\tens{E}\|_F\);
\(\sigma\) is the relative Frobenius noise level before final normalization.
In each trial, all methods share the tensor and a Haar-distributed initial
matrix independent of it.

The main setting is \(n=8,\rho=0.2,\sigma=0.3\).
We vary \(\sigma\in\{0.1,\allowbreak0.2,\allowbreak0.3,\allowbreak0.4\}\)
at \(n=8,\rho=0.2\), and
\(n\in\{4,\allowbreak8,\allowbreak12,\allowbreak16,\allowbreak20,\allowbreak24\}\)
at \(\sigma=0.3\) with \(\rho_n=1.4/(n-1)\).
The dimension sweep thus fixes the expected average degree at \(1.4\).
Each setting uses twenty paired evaluation trials. The main-setting trials
are reused in both sweeps. Across noise levels, each trial uses the same core,
noise direction and initial matrix.

We compare the same four Jacobi variants as in the preceding ICA experiment,
using the same thresholds and subproblem parameters, with RSG
\cite{li2021weakly} and \(\ell_1\)-PAM \cite{li2026ell}.
RSG uses polar retraction and \(\alpha_k=c/k^p\), \(k\geq1\), with
\(c\in\{0.01,\allowbreak0.05,\allowbreak0.1,\allowbreak0.5,\allowbreak1,\allowbreak5,\allowbreak10,\allowbreak20\}\)
and \(p\in\{0.5,1\}\). \(\ell_1\)-PAM uses the Hessian-based rule of
\cite{li2026ell}, with
\(\alpha_0\in\{10^{-4},\allowbreak10^{-3},\allowbreak10^{-2},\allowbreak10^{-1},\allowbreak1,\allowbreak5,\allowbreak9,\allowbreak13\}\).
We select each baseline's parameters once on twenty independent calibration
instances at the main setting, maximizing the mean final normalized
objective among candidates that converge on all calibration instances.
We use the selected values \(c=1,p=0.5\) for RSG and
\(\alpha_0=10^{-4}\) for \(\ell_1\)-PAM at every setting in both sweeps.

We report mean final normalized objective
\(f_{\mathrm{norm}}(\matr{U})=f(\matr{U})/(\sqrt n\|\tens{A}\|_F)\)
and median solver time.
The common stopping condition is
\(\|\matr{G}_{\mathrm R}(\matr{U})\|_F/\|\tens{A}\|_F\leq10^{-4}\),
using the real specialization of \eqref{eq:selected-riemannian-subgradient}.
At a nonsmooth point, this selected subgradient norm need not equal the
minimum norm of a Riemannian subgradient. Jacobi-CyclicB allows at most 2,000 sweeps; the Jacobi-SubG variants
allow 2,000 blocks of \(n(n-1)/2\) pair updates. RSG and \(\ell_1\)-PAM allow
50,000 and 5,000 iterations, respectively.

All evaluation trials of the Jacobi variants and \(\ell_1\)-PAM satisfy the
stopping criterion. RSG also satisfies it at the main setting, but does not
reach the threshold in some trials at lower noise levels and larger dimensions.

\Cref{fig:exp2-tensor} shows that, at the main setting, all four Jacobi variants
attain higher mean final objectives than \(\ell_1\)-PAM and RSG, with
Jacobi-CyclicB attaining the highest mean. At larger dimensions,
Jacobi-SubG-Max and Jacobi-SubG(\(1/\sqrt{n(n-1)}\)) have lower median
solver times than Jacobi-SubG(\(10^{-4}\)).

\begin{figure}[htbp]
\centering
\includegraphics[width=\linewidth]{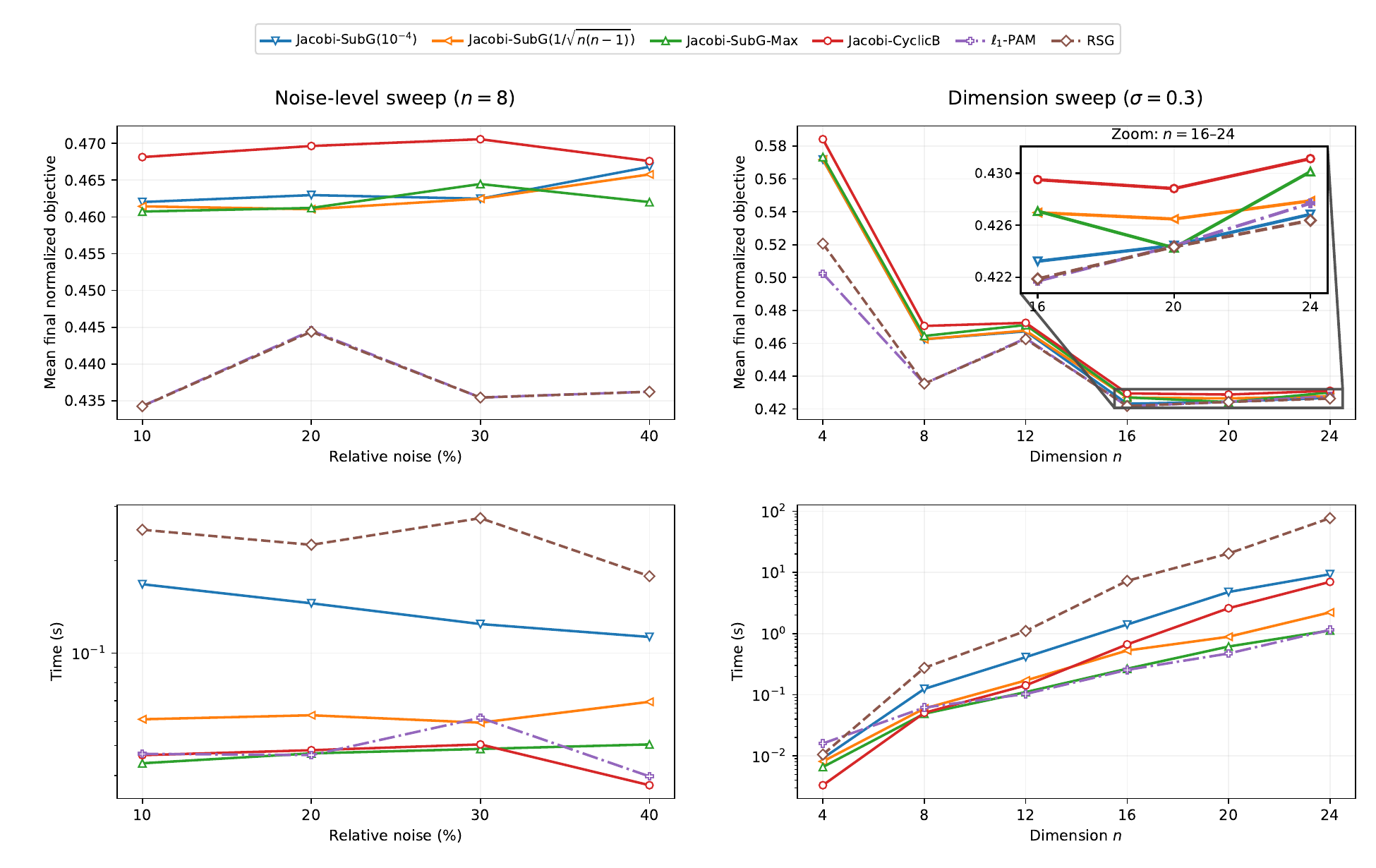}
\caption{Comparison of mean final normalized objectives and median solver times for tensor approximate diagonalization over twenty paired trials per setting.}
\label{fig:exp2-tensor}
\end{figure}

\section{Conclusion}\label{sec:conclusion}

Motivated by the widespread use of cumulant tensors in ICA, we develop
the Jacobi-SubG framework and establish global convergence results for
the general tensor model \eqref{cost-fn-complex-l1-diag} and the third-
and fourth-order cumulant models \eqref{eq:hermit_sy_order3-new} and
\eqref{eq:hermit_sy_order4-new}.
The numerical experiments show that the proposed Jacobi methods achieve better
source recovery and higher final objective values in several tested
settings, but do not offer a consistent runtime advantage.
The batched implementation, Jacobi-CyclicB, reduces the cost of cyclic sweeps.
Future work will focus on parallel and randomized strategies to accelerate
Jacobi-type algorithms while preserving their favorable recovery
performance and solution quality.

\appendix
\crefalias{section}{appendix}
\section{Proofs for the Jacobi-SubG framework}
\label{app:jacobi-subg-proofs}

\begin{proof}[Proof of \cref{prop:admissible-pair}]
Fix $\matr{U}\in\UN{n}$. By \eqref{eq:riemannian-subdiff} and \eqref{eq:tensor-subdiff},
the Riemannian subdifferential is obtained from coefficients
$\alpha_p\in\set{S}(g_p(\matr{U}))$.
This set is nonempty, compact, and convex. Choose its minimum-norm
element $\matr{\zeta}^\ast$, together with coefficients
$\alpha_p^\ast$ realizing it. Thus
$\|\matr{\zeta}^\ast\|_F
=
\dist\bigl(\vect{0},\partial_{\mathcal R}f(\matr{U})\bigr)$.
For a diagonal unitary matrix
$\matr{D}=\operatorname{diag}(e^{\mathrm i\phi_1},\ldots,
e^{\mathrm i\phi_n})$, phase covariance gives
$g_p(\matr{U}\matr{D})
=e^{\mathrm i(d-2r)\phi_p}g_p(\matr{U})$,
and hence $f(\matr{U}\matr{D})=f(\matr{U})$. For
$\matr{Z}_{\vect{\omega}}
=\matr{U}\,\mathrm i\operatorname{diag}(\omega_1,\ldots,\omega_n)$,
\[
\begin{aligned}
Dg_p(\matr{U})[\matr{Z}_{\vect{\omega}}]
=\mathrm i(d-2r)\omega_pg_p(\matr{U}),\ \
\langle\matr{\zeta}^\ast,\matr{Z}_{\vect{\omega}}\rangle
=\sum_{p=1}^n
\operatorname{Re}\!\left(
\overline{\alpha_p^\ast}\,
\mathrm i(d-2r)\omega_pg_p(\matr{U})
\right)=0.
\end{aligned}
\]
Indeed, if $g_p(\matr{U})\ne0$, then
$\alpha_p^\ast=g_p(\matr{U})/|g_p(\matr{U})|$; if
$g_p(\matr{U})=0$, the corresponding directional derivative
vanishes. Hence $\matr{\zeta}^\ast$ is orthogonal
to the phase tangent space.

For each pair $i<j$, let
$\phi_{ij}(z)=
\matr{U}\matr{G}^{(i,j,\matr{\Psi}(z))}$,
where $z\in\CC$ with $|z|<1$, $s=z$,
and $c=\sqrt{1-|z|^2}$.
Set $\widetilde h_{ij}=f\circ\phi_{ij}$.
As in \cref{sec:jacobi-subg}, we write
$\partial h_{(i,j),\matr{U}}(\matr{I}_2)$ for
$\partial\widetilde h_{ij}(0)$.
Applying the same exact chain rule to
$\ell\circ\vect{g}\circ\phi_{ij}$ yields
\begin{align*}
\partial\widetilde h_{ij}(0)
&=
\left\{
\sum_{p=1}^n
D(g_p\circ\phi_{ij})(0)^*[\alpha_p]:
\alpha_p\in\set{S}(g_p(\matr{U}))
\right\},\\
\vect{\xi}_{ij}
&\eqdef
D\phi_{ij}(0)^*\matr{\zeta}^\ast
=
\sum_{p=1}^n
D(g_p\circ\phi_{ij})(0)^*[\alpha_p^\ast]
\in\partial\widetilde h_{ij}(0).
\end{align*}
Thus the same coefficients $\alpha_p^\ast$ produce a valid local
subgradient for every pair.
Set
$\matr{\Omega}_{ij}(z)=
\overline z\,\vect{e}_j\vect{e}_i^\T
-z\,\vect{e}_i\vect{e}_j^\T$.
Then
$D\phi_{ij}(0)[z]=\matr{U}\matr{\Omega}_{ij}(z)$
and $\|\matr{\Omega}_{ij}(z)\|_F^2=2|z|^2$.
If $\operatorname{proj}_{ij}$ denotes the orthogonal projection
onto the corresponding Jacobi tangent block, then
$\|\vect{\xi}_{ij}\|^2
=2\|\operatorname{proj}_{ij}\matr{\zeta}^\ast\|_F^2$.
These blocks are mutually orthogonal and, together with the phase
tangent space, form an orthogonal decomposition of
$\mathrm T_{\matr{U}}\UN{n}$.
Since $\matr{\zeta}^\ast$ has no phase component,
\[
\sum_{i<j}\|\vect{\xi}_{ij}\|^2
=
2\sum_{i<j}
\|\operatorname{proj}_{ij}\matr{\zeta}^\ast\|_F^2
=
2\|\matr{\zeta}^\ast\|_F^2.
\]
There are $n(n-1)/2$ pairs, so at least one satisfies
\[
\|\vect{\xi}_{ij}\|
\geq
\frac{2}{\sqrt{n(n-1)}}\|\matr{\zeta}^\ast\|_F
\geq
\delta\,
\dist\bigl(\vect{0},\partial_{\mathcal R}f(\matr{U})\bigr),
\]
where the last inequality follows from the definition of
$\matr{\zeta}^\ast$ and the prescribed bound on $\delta$.
This proves \eqref{inequality-gra-based-2}.
\end{proof}

\begin{proof}[Proof of \cref{lem:nonvanishing-safeguard}]
Only the refinement branch requires verification; there
$\widehat{\Delta}_k>0$.
For $p\in\{i_k,j_k\}$, the function
$t\mapsto|g_p(\matr{U}_k(t))|^2$ is real analytic on a neighborhood
of $[0,1]$ and positive at $t=0$.
Its zeros are therefore isolated, so both selected contractions
are nonzero for $t<1$ sufficiently close to $1$.
Continuity gives
$h_k(\matr{\Psi}_k(t))-h_k(\matr{I}_2)
\geq\tfrac12\widehat{\Delta}_k$ for such $t$.
Thus the search terminates, and the other contractions are unchanged.
The gain bound holds by construction, while
$\|\matr{\Psi}_k(t)-\matr{I}_2\|_F
=2\sqrt2\sin(t\widehat\theta_k/2)$ is nondecreasing on $[0,1]$.
The zero-gain assertion follows from the first branch of the safeguard.
\end{proof}

\begin{proof}[Proof of \cref{thm:jacobi-conditional}]
By \cref{lem:nonvanishing-safeguard}, every iterate has nonzero
diagonal contractions. Thus $f$ is smooth near $\matr{U}_{k-1}$,
and $h_k$ is smooth near $\matr{I}_2$.
Since $\widetilde h_k$ is smooth at the origin,
$\partial h_k(\matr{I}_2)$ is a singleton, and its element
$\vect{\xi}_k$ satisfies
$\dist(\vect{0},\partial h_k(\matr{I}_2))
=\|\vect{\xi}_k\|
\geq\delta\|\grad f(\matr{U}_{k-1})\|_F$.
Combining \eqref{eq:subproblem_sufficient_ascent} with
\cref{lem:nonvanishing-safeguard} and
$\|\matr{U}_k-\matr{U}_{k-1}\|_F
=\|\matr{\Psi}_k-\matr{I}_2\|_F$ gives, for all sufficiently large $k$,
\begin{equation}\label{eq:global-sufficient-ascent}
f(\matr{U}_k)-f(\matr{U}_{k-1})
\geq
a\|\grad f(\matr{U}_{k-1})\|_F
\|\matr{U}_k-\matr{U}_{k-1}\|_F,
\qquad
a:=\frac{\varepsilon_{\mathrm{sub}}\delta}{2}>0.
\end{equation}

Regard the ambient complex matrix space as a real Euclidean space
and set $\varphi=-f$.
Compactness of $\UN{n}$ supplies an accumulation point
$\matr{U}_\ast\in\UN{n}$.
The function $\Phi=\varphi+\delta_{\UN{n}}$ is proper, lower
semicontinuous, semialgebraic, and continuous on its closed domain,
with a locally Lipschitz restriction to $\UN{n}$.
The nonsmooth {\L}ojasiewicz inequality
\cite{bolte2007lojasiewicz}, together with
$\dist(\vect{0},\partial\Phi(\matr{U}))
=\dist(\vect{0},\partial_{\mathcal R}\varphi(\matr{U}))$
for $\matr{U}\in\UN{n}$, therefore yields
\eqref{eq:riemannian-lojasiewicz} at $\matr{U}_\ast$
with some exponent $\mu\in[0,1)$. 
At every iterate, smoothness gives
$\partial_{\mathcal R}\varphi(\matr{U}_k)
=\{-\grad f(\matr{U}_k)\}$.
Hence \eqref{eq:global-sufficient-ascent} gives
\textup{(A1)} of \cref{prop:riemannian-KL-convergence}
with $\alpha=\beta=1$.
Condition \textup{(A2)} is automatic with $\eta=0$, and
\textup{(A3)} follows from the zero-gain rule in
\cref{lem:nonvanishing-safeguard}.
Since $(\alpha+\eta)\mu=\mu<1$,
\cref{prop:riemannian-KL-convergence} proves that
$\matr{U}_k\to\matr{U}_\ast$.
\end{proof}

\begin{proof}[Proof of \cref{thm:jacobi-specialized}]
Both candidate generators are nondecreasing.
The fourth-order solver maximizes an active branch, possibly minus
a nonnegative penalty vanishing at the identity.
The third-order solver starts with maximizing phases and performs
nondecreasing auxiliary updates; the reduced objective dominates
the returned auxiliary value.
By \cref{lem:nonvanishing-safeguard}, every current point has
nonzero diagonal contractions.
The proofs of
\cref{prop:uniform-ascent-complex4,prop:uniform-ascent-complex3}
require only this pointwise nonvanishing, and their constants
are independent of the contraction magnitudes.
Since the full and reduced objectives differ only by $C_{k,\mathsf X}$,
these estimates yield \eqref{eq:subproblem_sufficient_ascent}
for the candidates. Apply \cref{thm:jacobi-conditional}.
\end{proof}

\noindent\textbf{Acknowledgments.}
This work was supported by the Guangdong Basic and Applied Basic
Research Foundation (2026A1515050001, 2026A1515010188).

\bibliographystyle{siamplain}
\bibliography{ReferenceTensor.bib}

\begin{thebibliography}{10}

\bibitem{AbsMA05:sjo}
{\sc P.-A. Absil, R.~Mahony, and B.~Andrews}, {\em Convergence of the iterates
  of descent methods for analytic cost functions}, SIAM Journal on
  Optimization, 16 (2005), pp.~531--547.

\bibitem{attouch2013convergence}
{\sc H.~Attouch, J.~Bolte, and B.~F. Svaiter}, {\em Convergence of descent
  methods for semi-algebraic and tame problems: proximal algorithms,
  forward--backward splitting, and regularized gauss--seidel methods},
  Mathematical Programming, 137 (2013), pp.~91--129.

\bibitem{begovic2024trace}
{\sc E.~Begovi{\'c}~Kova{\v{c}} and A.~Perkovi{\'c}}, {\em Trace maximization
  algorithm for the approximate tensor diagonalization}, Linear and multilinear
  algebra, 72 (2024), pp.~429--450.

\bibitem{bolte2007lojasiewicz}
{\sc J.~Bolte, A.~Daniilidis, and A.~Lewis}, {\em The {\l}ojasiewicz inequality
  for nonsmooth subanalytic functions with applications to subgradient
  dynamical systems}, SIAM Journal on Optimization, 17 (2007), pp.~1205--1223.

\bibitem{bolte2014proximal}
{\sc J.~Bolte, S.~Sabach, and M.~Teboulle}, {\em Proximal alternating
  linearized minimization for nonconvex and nonsmooth problems}, Mathematical
  Programming, 146 (2014), pp.~459--494.

\bibitem{boumal2023intromanifolds}
{\sc N.~Boumal}, {\em An introduction to optimization on smooth manifolds},
  Cambridge University Press, 2023.

\bibitem{cardoso1993blind}
{\sc J.-F. Cardoso and A.~Souloumiac}, {\em Blind beamforming for non-gaussian
  signals}, in IEE proceedings F (radar and signal processing), vol.~140, IET,
  1993, pp.~362--370.

\bibitem{cardoso1996jacobi}
{\sc J.-F. Cardoso and A.~Souloumiac}, {\em Jacobi angles for simultaneous
  diagonalization}, SIAM Journal on Matrix Analysis and Applications, 17
  (1996), pp.~161--164.

\bibitem{comon1994independent}
{\sc P.~Comon}, {\em Independent component analysis, a new concept?}, Signal
  processing, 36 (1994), pp.~287--314.

\bibitem{Como10:book}
{\sc P.~Comon and C.~Jutten}, eds., {\em Handbook of Blind Source Separation},
  Academic Press, Oxford, 2010.

\bibitem{comon1997improved}
{\sc P.~Comon and E.~Moreau}, {\em Improved contrast dedicated to blind
  separation in communications}, in 1997 IEEE International Conference on
  Acoustics, Speech, and Signal Processing, vol.~5, IEEE, 1997, pp.~3453--3456.

\bibitem{comon2007tensor}
{\sc P.~Comon and M.~Sorensen}, {\em Tensor diagonalization by orthogonal
  transforms}, Report ISRN I3S-RR-2007-06-FR,  (2007).

\bibitem{de1997signal}
{\sc L.~De~Lathauwer}, {\em Signal processing based on multilinear algebra},
  Katholieke Universiteit Leuven, 1997.

\bibitem{Lathauwer96:simultaneous}
{\sc L.~{De Lathauwer}, B.~{De Moor}, and J.~Vandewalle}, {\em Blind source
  separation by simultaneous third-order tensor diagonalization}, in 1996 8th
  European Signal Processing Conference (EUSIPCO 1996), 1996, pp.~1--4.

\bibitem{drusvyatskiy2019efficiency}
{\sc D.~Drusvyatskiy and C.~Paquette}, {\em Efficiency of minimizing
  compositions of convex functions and smooth maps}, Mathematical Programming,
  178 (2019), pp.~503--558, \url{https://doi.org/10.1007/s10107-018-1311-3}.

\bibitem{ge2008analysis}
{\sc F.~Ge and J.~Ma}, {\em Analysis of the kurtosis-sum objective function for
  ica}, in International Symposium on Neural Networks, Springer, 2008,
  pp.~579--588.

\bibitem{ge2010efficient}
{\sc F.~Ge and J.~Ma}, {\em An efficient pairwise kurtosis optimization
  algorithm for independent component analysis}, in International Conference on
  Intelligent Computing, Springer, 2010, pp.~94--101.

\bibitem{huang2022riemannian}
{\sc W.~Huang and K.~Wei}, {\em {R}iemannian proximal gradient methods},
  Mathematical Programming, 194 (2022), pp.~371--413.

\bibitem{hyvarinen1999fast}
{\sc A.~Hyvarinen}, {\em Fast and robust fixed-point algorithms for independent
  component analysis}, IEEE transactions on Neural Networks, 10 (1999),
  pp.~626--634.

\bibitem{Hyvarinen2001}
{\sc A.~Hyv{\"a}rinen, J.~Karhunen, and E.~Oja}, {\em Independent Component
  Analysis}, John Wiley \& Sons, Ltd, 2001.

\bibitem{hyvarinen1997fast}
{\sc A.~Hyv{\"a}rinen and E.~Oja}, {\em A fast fixed-point algorithm for
  independent component analysis}, Neural computation, 9 (1997),
  pp.~1483--1492.

\bibitem{IshtAV13:simax}
{\sc M.~Ishteva, P.-A. Absil, and P.~{Van Dooren}}, {\em Jacobi algorithm for
  the best low multilinear rank approximation of symmetric tensors}, SIAM
  Journal on Matrix Analysis and Applications, 34 (2013), pp.~651--672.

\bibitem{karvanen2004independent}
{\sc J.~Karvanen and V.~Koivunen}, {\em Independent component analysis via
  optimum combining of kurtosis and skewness-based criteria}, Journal of the
  Franklin Institute, 341 (2004), pp.~401--418.

\bibitem{kolda2011shifted}
{\sc T.~G. Kolda and J.~R. Mayo}, {\em Shifted power method for computing
  tensor eigenpairs}, SIAM Journal on Matrix Analysis and Applications, 32
  (2011), pp.~1095--1124.

\bibitem{lageman2002konvergenz}
{\sc C.~Lageman}, {\em Konvergenz reell-analytischer gradienten{\"a}hnlicher
  systeme}, Master's thesis. Germany: Institut f{\"u}r Mathematik,
  Universit{\"a}t W{\"u}rzburg,  (2002).

\bibitem{ledyaev2007nonsmooth}
{\sc Y.~Ledyaev and Q.~Zhu}, {\em Nonsmooth analysis on smooth manifolds},
  Transactions of the American Mathematical Society, 359 (2007),
  pp.~3687--3732.

\bibitem{LUC2017globally}
{\sc J.~Li, K.~Usevich, and P.~Comon}, {\em Globally convergent {J}acobi-type
  algorithms for simultaneous orthogonal symmetric tensor diagonalization},
  SIAM Journal on Matrix Analysis and Applications, 39 (2018), pp.~1--22.

\bibitem{li2023convergence}
{\sc J.~Li, K.~Usevich, and P.~Comon}, {\em Convergence of gradient-based block
  coordinate descent algorithms for nonorthogonal joint approximate
  diagonalization of matrices}, SIAM Journal on Matrix Analysis and
  Applications, 44 (2023), pp.~592--621.

\bibitem{li2019polar}
{\sc J.~Li and S.~Zhang}, {\em Polar decomposition-based algorithms on the
  product of {S}tiefel manifolds with applications in tensor approximation},
  Journal of the Operations Research Society of China,  (2024), p.~874–920.

\bibitem{li2021weakly}
{\sc X.~Li, S.~Chen, Z.~Deng, Q.~Qu, Z.~Zhu, and A.~Man-Cho~So}, {\em Weakly
  convex optimization over {S}tiefel manifold using {R}iemannian
  subgradient-type methods}, SIAM Journal on Optimization, 31 (2021),
  pp.~1605--1634.

\bibitem{li2026ell}
{\sc X.~Li and Y.~Yang}, {\em {$\ell_1$}-norm simultaneous approximate
  diagonalization and a proximal alternating maximization method}, Journal of
  Computational and Applied Mathematics, 483 (2026), p.~117421,
  \url{https://doi.org/10.1016/j.cam.2026.117421}.

\bibitem{li2026robust}
{\sc X.~Li and Y.~Yang}, {\em A robust {$\ell_1$}-norm approach to simultaneous
  approximate tensor diagonalization via {Riemannian} smoothing}, Journal of
  Applied Mathematics and Computing, 72 (2026), 143,
  \url{https://doi.org/10.1007/s12190-026-02800-4}.

\bibitem{lojasiewicz1963propriete}
{\sc S.~Lojasiewicz}, {\em Une propri{\'e}t{\'e} topologique des sous-ensembles
  analytiques r{\'e}els}, Les {\'e}quations aux d{\'e}riv{\'e}es partielles,
  117 (1963).

\bibitem{law1965ensembles}
{\sc S.~{\L}ojasiewicz}, {\em Ensembles semi-analytiques}, IHES notes,  (1965).

\bibitem{lojasiewicz1982trajectoires}
{\sc S.~Lojasiewicz}, {\em Sur les trajectoires du gradient d’une fonction
  analytique}, Seminari di geometria, 1983 (1982), pp.~115--117.

\bibitem{martin2008jacobi}
{\sc C.~D.~M. Martin and C.~F. Van~Loan}, {\em A {J}acobi-type method for
  computing orthogonal tensor decompositions}, SIAM Journal on Matrix Analysis
  and Applications, 30 (2008), pp.~1219--1232.

\bibitem{moreau1996criteria}
{\sc E.~Moreau}, {\em Criteria for complex sources separation}, in 1996 8th
  European Signal Processing Conference (EUSIPCO 1996), IEEE, 1996, pp.~1--4.

\bibitem{moreau2002generalization}
{\sc E.~Moreau}, {\em A generalization of joint-diagonalization criteria for
  source separation}, IEEE Transactions on Signal Processing, 49 (2001),
  pp.~530--541.

\bibitem{moreau1996high}
{\sc E.~Moreau and O.~Macchi}, {\em High-order contrasts for self-adaptive
  source separation}, International Journal of Adaptive Control and Signal
  Processing, 10 (1996), pp.~19--46.

\bibitem{papadias2002globally}
{\sc C.~B. Papadias}, {\em Globally convergent blind source separation based on
  a multiuser kurtosis maximization criterion}, IEEE Transactions on Signal
  processing, 48 (2000), pp.~3508--3519.

\bibitem{rockafellar1998variational}
{\sc R.~T. Rockafellar and R.~J. Wets}, {\em Variational analysis}, Springer,
  1998.

\bibitem{SU15:pro}
{\sc R.~Schneider and A.~Uschmajew}, {\em Convergence results for projected
  line-search methods on varieties of low-rank matrices via {\l{}}ojasiewicz
  inequality}, SIAM Journal on Optimization, 25 (2015), pp.~622--646.

\bibitem{sheng2023jacobi}
{\sc Z.~Sheng, J.~Li, and Q.~Ni}, {\em Jacobi-type algorithms for homogeneous
  polynomial optimization on stiefel manifolds with applications to tensor
  approximations}, Mathematics of Computation, 92 (2023), pp.~2217--2245.

\bibitem{song2016ecoica}
{\sc L.~Song and H.~Lu}, {\em Ecoica: Skewness-based ica via eigenvectors of
  cumulant operator}, in Asian Conference on Machine Learning, PMLR, 2016,
  pp.~445--460.

\bibitem{ULC2019}
{\sc K.~Usevich, J.~Li, and P.~Comon}, {\em Approximate matrix and tensor
  diagonalization by unitary transformations: convergence of {J}acobi-type
  algorithms}, SIAM Journal on Optimization, 30 (2020), pp.~2998--3028.

\bibitem{wang2023linear}
{\sc P.~Wang, H.~Liu, and A.~M.-C. So}, {\em Linear convergence of a proximal
  alternating minimization method with extrapolation for {$\ell_1$}-norm
  principal component analysis}, SIAM Journal on Optimization, 33 (2023),
  pp.~684--712.

\bibitem{zarzoso2009robust}
{\sc V.~Zarzoso and P.~Comon}, {\em Robust independent component analysis by
  iterative maximization of the kurtosis contrast with algebraic optimal step
  size}, IEEE Transactions on neural networks, 21 (2009), pp.~248--261.

\end{thebibliography}

\end{document}